\documentclass[12pt,a4paper]{article}%
\usepackage{amssymb}
\usepackage{amsmath}
\usepackage{amsfonts}
\usepackage[dvips]{graphicx}
\usepackage{mathrsfs}
\usepackage[pdftex]{hyperref}
\usepackage{pdfpages}%
\usepackage[all]{xy}
\providecommand{\U}[1]{\protect\rule{.1in}{.1in}}
\hypersetup{
	colorlinks=true,
	citecolor = blue
}
\newtheorem{theorem}{Theorem}

\newtheorem{corollary}[theorem]{Corollary}

\newtheorem{definition}[theorem]{Definition}

\newtheorem{lemma}[theorem]{Lemma}

\newtheorem{proposition}[theorem]{Proposition}
\newtheorem{remark}[theorem]{Remark}

\newenvironment{proof}[1][Proof]{\noindent\textbf{#1.} }{\ \rule{0.5em}{0.5em}}

\begin{document}
	
	\author{V\'{\i}ctor A. Vicente-Ben\'{\i}{}tez\\{\small Instituto de Matemáticas de la U.N.A.M. Campus Juriquilla}\\{\small Boulevard Juriquilla 3001, Juriquilla, Querétaro C.P. 076230 M\'{e}xico }\\ {\small  va.vicentebenitez@im.unam.mx } }
	\title{Complete systems of solutions, transmutations and Darboux transform for Sturm-Liouville equations in impedance form}
	\date{}
	\maketitle

	\begin{abstract}
		We present the construction of a complete system of functions associated with the Sturm–Liouville equation in impedance form on a finite interval $I$, given an impedance function $a\in L^2(I)$. The system, known as the {\it formal powers}, is generated through recursive integration of the impedance function $a$ and its reciprocal. We establish the completeness of this system in the space $L^p$ with the weight function $a^2$. Under additional conditions on $a$, we extend this completeness to Sobolev spaces $ W^ {1,p}(I)$, along with a generalized Taylor formula. 
		
		We show that the completeness of the formal powers implies key analytic properties for a transmutation operator associated with the Sturm-Liouville equation in impedance form, including the existence of a continuous inverse. Finally, we introduce a formulation of the Darboux-transformed equation and establish a relation between the transmutation operators to the original and transformed equations.
	\end{abstract}
	
	\newline 
	
	
	\section{Introduction}
	
	In this work, we present the construction of a complete system of functions associated with the Sturm-Liouville equation in the form
	\begin{equation}\label{eq:impedanceintro}
		-\frac{d}{dx}\left(a^2(x)\frac{du(x)}{dx}\right)=\lambda a^2(x)u(x),\quad x\in I, \lambda\in \mathbb{C},
	\end{equation}
	where $I\subset \mathbb{R}$ is a finite interval, $\lambda$ is called the {\it spectral parameter} of Eq. \eqref{eq:impedanceintro}, and the coefficient $a$ is referred to as the {\it impedance function}. We assume that $a\in L^2(I)$ is positive and satisfies $\frac{1}{a}\in L^2(I)$, such functions are often called \textit{proper impedance functions}. A general method for solving Sturm–Liouville equations was introduced, known as the spectral parameter power series (SPPS) method. This approach was initially developed for Sturm–Liouville equations with continuous coefficients, and later extended in \cite{sppscampos} to equations with locally integrable coefficients. The SPPS method constructs solutions to Eq.~\eqref{eq:impedanceintro} by expanding them into Taylor series with respect to the spectral parameter. This expansion yields a sequence of Taylor coefficients $\{\varphi_k\}_{k=0}^{\infty}$ which are functions of the variable $x$. In the SPPS literature, these coefficients are referred to as formal powers \cite{camposlbases}, and are closely related to the concept of an $L$-basis for a differential operator, as introduced in \cite{fage}. The formal powers associated with a regular differential equation generalize several key properties of standard monomials, including completeness in certain functional spaces. For instance, in \cite{kravchenkocompleteness}, it was shown that the formal powers corresponding to a Schr\"odinger equation with a continuous potential form a complete system in $L^2(a,b)$ and $C[a,b]$. Furthermore, in \cite{tremblay}, various properties of these formal powers—such as binomial-type expansions and generalized trigonometric identities—are investigated. The completeness property of formal powers also facilitates the construction of complete systems of solutions for partial differential equations that can be reduced, via separation of variables, to the Schr\"odinger equation \cite{camposlbases, josafath, minerunge}. One of the key properties of formal powers is that their completeness ensures the existence of a continuous transmutation operator, that is, a Volterra integral operator that transforms the solutions of the equation $v''+\lambda v=0$ into the solutions of the Schr\"odinger equation \cite{marchenko,povzner}. Moreover, in \cite{camposlbases} it was shown that the formal powers are the image of the standard monomials under the transmutation operator. This property is known as the {\it transmutation property of the formal powers}, and was extended for Schr\"odinger equations with integrable potentials in \cite{camposstandard}, where it was shown that certain transmutation operators are completely determined by the corresponding associated formal powers. The significance of this transmutation property lies in that it allows us to construct analytical approximations, both for the transmutation operator itself and for the solutions of the equation to be studied, without requiring an explicit expression for the full operator \cite{NSBF1, kravchenkotorba1}. Transmutation operators have proven to be powerful tools in the practical solution of direct and inverse spectral problems,  particularly in the context of the Schr\"odinger equation, as developed over the past decade (see, e.g., \cite{kravchenkolibro, NSBF1, spps, kravchenkotorbainverse}). For the impedance-form equation, the existence of certain integral representations for its solutions was established in \cite{carroll}, while the fundamental properties of a transmutation operator in the case where $a\in C^1[0,\ell]$, was investigated in \cite{mineimpedance1}. It is worth mentioning that, even though for an impedance $a \in W^{1,2}$, the impedance equation is unitarily equivalent to the Schrödinger equation via the Liouville transformation, and one can thus apply the transmutation theory developed for the latter, if the impedance does not possess a second weak derivative in $ L^p$, the associated Schrödinger equation involves a distributional potential. There are some works concerning the existence of transmutation operators for certain classes of distributional potentials \cite{minedelta, shkalikov}, as well as integral representations for solutions of Eq. \eqref{eq:impedanceintro} satisfying specific initial conditions \cite{albeverio}. However, it is of particular interest to develop a transmutation theory directly for the Sturm-Liouville equation in impedance form, as this allows for explicit analytical representations and provides tools for the efficient solution of both direct and inverse spectral problems.
	
	This work aims to extend the construction of the formal powers to the impedance equation \eqref{eq:impedanceintro} with a proper integrable impedance. The system is defined by
	\begin{equation*}
		\varphi_a^{(0)}\equiv 1, \;\; \varphi_a^{(1)}(x):=\int_{x_0}^x\frac{dt}{a^2(t)},\;\;  \varphi_a^{(k)}(x):=k(k-1)\int_{x_0}^{x}\frac{dt}{a^2(t)}\int_{x_0}^ta^2(s)\varphi_a^{(k-2)}(s)ds, \; k\geq 2,
	\end{equation*}
	where $x_0\in \overline{I}$. We show that the formal powers $\{\varphi_a^{(k)}\}_{k=0}^{\infty}$ satisfy recursive relations involving the differential operators $\mathbf{D}_{a^{(-1)^k}}:=(a^2)^{(-1)^k}\frac{d}{dx}$, analogous to the relations between the standard monomials and the classical derivative operator. Moreover, we establish conditions for the existence of a generalized Taylor formula for solutions of the equation $\mathbf{D}^{(m)}_au=g$, where $\mathbf{D}^{(m)}u:=\prod_{j=1}^{m}\mathbf{D}_{a^{(-1)^{(j-1)}}}$.  We prove the completeness of the formal powers in the $L^p$-space with weight function $a^2$, and under additional hypotheses (that we establish in Sec. \ref{Sec: Completeness}), the completeness in Sobolev spaces $W^{m,p}(I)$. For completeness in $L^p$, we first consider a Sturm-Liouville problem with regular boundary conditions, associated to Eq. \eqref{eq:impedanceintro}. Using the SPPS method, we demonstrate that the system of formal powers is complete in the eigenfunction space of the associated problem. Subsequently, we study the completeness of the eigenfunctions in $L^p$ spaces, obtaining the completeness of the formal powers by transitivity. For completeness in $W^{1,p}$, we analyze the conditions under which a function admits a generalized Taylor approximation in terms of formal powers.
	
	As an application, we show how the completeness of the formal powers enables the construction of a transmutation operator for Eq. \eqref{eq:impedanceintro} when the impedance function belongs to $W^{1,\infty}$, along with some analytical properties such as the existence of a continuous inverse. Similarly, we will analyze the relationship between the formal powers for Eq. \eqref{eq:impedanceintro} and those corresponding powers for the equation with impedance $a^{-1}$, which, following \cite{mineimpedance1}, we refer to as {\it the associated Darboux equation}. These connections allow us to derive relations between the transmutation operator of Eq. \eqref{eq:impedanceintro} and its associated Darboux equation. It is worth mentioning that the existence of relations between the transmutation operator and the associated Darboux operator plays a crucial role in constructing transmutation operators and complete systems of solutions for systems of partial equations, such as Vekua-type equations with separable coefficients \cite{camposlbases,kravchenkotorbahyperbolic,minevekua1}. Finally, we study the existence of a transmutation operator for the case where $a\in W^{1,2}(0,\ell)$, along with its analytical properties.
	
	The paper is structured as follows. Section 2 provides background on the main properties of the solutions and factorizations of the impedance operator. Section 3 develops the definition and key properties of the formal powers, generalized derivatives, and formal Taylor polynomials.  Section 4 is devoted to proving the completeness of the formal powers in the $L^p$ and Sobolev spaces. In Section 5, we discuss the main properties of a transmutation operator for the case when $a\in W^{1,\infty}(-\ell,\ell)$, including the existence of a continuous inverse operator, and the relations with the transmutation operator for the Darboux-transformed equation. Finally, Section 6 presents the construction of a transmutation operator for the case $a\in W^{1,2}(0,\ell)$.

	\section{Types of solutions, factorization, Darboux transformation, and formal powers}
	
	Through the text, we use the notation $\mathbb{N}_0:=\mathbb{N}\cup\{0\}$.  Given Banach spaces $X$ and $Y$, $\mathcal{B}(X,Y)$ denotes the space of bounded linear operators. When $X=Y$, we write $\mathcal{B}(X):=\mathcal{B}(X,X)$, and the group of invertible operators with bounded inverse is denoted by $\mathcal{G}(X)$.
	
	Let $I\subset \mathbb{R}$ be a bounded interval and let $a\in L^2(I)$ be a positive measurable function such that $\frac{1}{a}\in L^2(I)$. A function satisfying this condition will be called a {\bf proper impedance}.
	
	For $1\leq p<\infty$, we use the notation $L_a^p(I)$ for the space of integrable measurable functions with respect to the weight function $a^2$. Hence, the norm is given by
	\[
	\|u\|_{L_a^p(I)}:=\left(\int_{I}|u(x)|^pa^2(x)dx\right)^{\frac{1}{p}}.
	\]
	For $p=2$, the corresponding inner product is given by
	\[
	\langle u,v\rangle_{L_a^2(I)}:=\int_{I}u(x)\overline{v(x)}a^2(x)dx.
	\]
	Since $a^2(x)dx$ defines a finite measure on $I$, the embedding $L^p_a(I)\hookrightarrow L_a^q(I)$ is continuous, whenever $q\leq p$ \cite[Prop. 6. 12]{folland}. If, in addition,  $a,\frac{1}{a}\in L^{\infty}(I)$, then $L_a^p(I)=L^p(I)$ with equivalent norms. We use the standard notation $W^{1,p}(I)$ for the Sobolev space of functions in $L^p(I)$ whose first $k$ distributional derivatives lie in $L^p(I)$. The space of absolutely continuous functions on $\overline{I}$ is denoted by $AC(\overline{I})$.

	The {\it Sturm-Liouville equation in impedance form} is given by
	
	\begin{equation}\label{eq:impedanceeq}
		-\frac{1}{a^2(x)}\frac{d}{dx}\left(a^2(x)\frac{du}{dx}\right)=\lambda u,\quad x\in I, \lambda \in \mathbb{C}.
	\end{equation}
	The function $a$ is called the {\it impedance} of Eq. \eqref{eq:impedanceeq}, and $\lambda$ is the spectral parameter. In some contexts, it is common to work with the {\it conductivity function} associated to $a$, that is, $\kappa(x):=a^2(x)$. Under our assumptions, we have that $\kappa,\frac{1}{\kappa}\in L^1(I)$ and $L_a^p(I)$ is just the $L^p$ space over $I$ with the measure $k(x)dx$.
	
	In general, we consider equations of the form
	\begin{equation}\label{eq:impedancegeneral}
		-\frac{1}{a^2(x)}\frac{d}{dx}\left(a^2(x)\frac{du}{dx}\right)=\lambda u+g(x),\quad x\in I, \lambda \in \mathbb{C},
	\end{equation}
	with $g\in L_a^p(I)$.  A {\it weak} solution of Eq. \eqref{eq:impedancegeneral}, is a function $u\in W^{1,1}(I)$ such that $u'\in L_a^p(I)$ and satisfies the following condition:
	\begin{equation}\label{eq:weakformulationimpedance}
		\int_{I}a^2(x)u'(x)\phi'(x)dx= \int_{I}a^2(x)\left(\lambda u(x)+g(x)\right)\phi(x)dx\qquad \forall \phi\in C_0^{\infty}(I).
	\end{equation}
	
	\begin{proposition}\label{Prop:Regularityofsol}
		If $u\in W^{1,1}(I)$ is a weak solution of an equation of type \eqref{eq:impedancegeneral}, then $a^2u\in AC(\overline{I})$ with $\frac{(a^2u')'}{a^2}\in L_a^p(I)$ and $u$ satisfies Eq. \eqref{eq:impedancegeneral} a.e. in $I$.
		
		If $a\in W^{1,p}(I)$, then $u\in W^{2,p}(I)$. If $a\in C^1(\overline{I})$ and $g\in C(\overline{I})$, then $u\in C^2(\overline{I})$ and is a classical solution.
	\end{proposition}
	\begin{proof}
		Suppose that $u\in W^{1,1}(I)$ is a weak solution of \eqref{eq:impedancegeneral}. Condition \eqref{eq:weakformulationimpedance} is equivalent to saying that $-a^2(\lambda u+g)$ is the distributional derivative of $a^2u'$. Hence $(a^2u')'=-a^2(\lambda u+g)\in L^p(I)$, which implies that $a^2u'\in AC(\overline{I})$ with $\frac{(a^2u')'}{a^2}=-(\lambda u+g)\in L_a^p(I)$. Therefore, $u$ satisfies \eqref{eq:impedancegeneral} a.e. in $I$ (see \cite[Th. 8.2]{brezis}).
		
		Now suppose that $a\in W^{1,p}(I)$. Hence $a\in AC(\overline{I})$ (see \cite[Th. 8.2]{brezis}) and $(\frac{1}{a})'=\frac{-a'}{a^2}\in L^p(I)$, so $\frac{1}{a}\in W^{1,p}(I)$ and we can assume that $\frac{1}{a}\in C(\overline{I})$. Consequently, $L^p(I)=L_a^p(I)$ with equivalent norms. Also note that $(a^2)'=2aa'\in L^p(I)$, so $a^2\in W^{1,p}(I)$. Hence $-a^2(\rho^2u+g)\in L^p(I)$ and $a^2u'\in W^{1,p}(I)$.  Thus, by \cite[Cor. 8.10]{brezis}, we have $u'=\frac{a^2u'}{a^2}\in W^{1,p}(I)$, so $u\in W^{2,p}(I)$.
		
		Finally, if $a',g\in C(\overline{I})$, then $u'=-\frac{2a'}{a}u'-a^2(\lambda u+g)\in C(\overline{I})$, and $u\in C^2(\overline{I})$ is a classical solution.
	\end{proof}

	Solutions $u\in W^{1,1}(I)$ such that $a^2u'\in AC(\overline{I})$, with $\frac{(a^2u')'}{a^2}\in L_a^p(I)$, and which satisfy the equation a.e. in $I$ are called {\it strong solutions}. The previous result then shows that any weak solution of Eq. \eqref{eq:impedancegeneral} is, in fact, a strong solution.
	
	Denote $D=\frac{d}{dx}$. Based on the previous results, we define the Sturm-Liouville operator $\mathbf{L}_a:\mathscr{D}_p(\mathbf{L}_a)\subset L_a^p(I) \rightarrow L_a^p(I)$ with domain consisting of functions $u\in W^{1,1}(I)$ and such that $a^2u'\in AC(\overline{I})$ with $\frac{(a^2u')'}{a^2}\in L_a^p(I)$, and with action given by
	\begin{equation}\label{eq:Impedancea}
		\mathbf{L}_{a}:=\frac{1}{a^2}Da^2D.
	\end{equation}
	Consider the differential operator $\mathbf{D}_a: \mathscr{D}_p(\mathbf{D}_a)\subset L_a^p(I)\rightarrow L_{\frac{1}{a}}^p(I)$ whose domain consists of functions $u\in W^{1,1}(I)$ with $a^2u'\in L_{\frac{1}{a}}^p(I)$ and with action given by
	\begin{equation}\label{eq:1storderdifferentialoperators}
		\mathbf{D}_a:= a^2D.
	\end{equation}
	Note that the operator $\mathbf{L}_a$ admits the factorization
	\begin{equation}
		\mathbf{L}_a=\mathbf{D}_{\frac{1}{a}}\mathbf{D}_a.
	\end{equation}
	A similar factorization holds for the differential operator $\mathbf{L}_{\frac{1}{a}}$. Actually, if $a\in W^{2,p}(I)$, we have the factorization
	\[
	\mathbf{L}_au=\mathbf{D}_{\frac{1}{a}}\mathbf{D}_a u= \frac{1}{a}\left(\frac{1}{a}Da\right)\left(aD\frac{1}{a}\right)(au)=\frac{1}{a}\left(D+\frac{a'}{a}\right)\left(D-\frac{a'}{a}\right)(au)=\frac{1}{a}\left(D^2-\frac{a''}{a}\right)(au).
	\]
	which involves the Schr\"odinger operator $\mathbf{S}_a:=D^2-\frac{a''}{a}$. In the case of $\mathbf{L}_{\frac{1}{a}}$, we obtain
	\[
	\mathbf{L}_{\frac{1}{a}}u=a\left(D^2-a\left(\frac{1}{a}\right)''\right)\left(\frac{u}{a}\right).
	\]
	The operator $\mathbf{S}_{\frac{1}{a}}=D^2-a\left(\frac{1}{a}\right)''$ is known as the {\it Darboux Sch\"odinger operator} associated with $\mathbf{S}_a$. For this reason, the operator $\mathbf{L}_{\frac{1}{a}}$ will be called the {\it Darboux impedance operator} associated with $\mathbf{L}_a$.
	
	\begin{remark}
		Let $u\in \mathscr{D}_p(\mathbf{L}_a)$ be a solution of Eq. \eqref{eq:impedanceeq}, and set $v=\mathbf{D}_au\in \mathscr{D}_p(\mathbf{D}_{\frac{1}{a}})$. Note that
		\[
		\mathbf{D}_{\frac{1}{a}}v=\mathbf{D}_{\frac{1}{a}}\mathbf{D}_au=\mathbf{L}_au=-\lambda u\in \mathscr{D}_p(\mathbf{D}_a).
		\]
		Consequently, $v$ belongs to the domain of the operator $\mathbf{L}_{\frac{1}{a}}$ and
		\[
		\mathbf{L}_{\frac{1}{a}}v=\mathbf{D}_a\mathbf{D}_{\frac{1}{a}}\mathbf{D}_au=-\lambda \mathbf{D}_au=-\lambda v.
		\]
		That is, $\mathbf{D}_a$ maps solutions of Eq. \eqref{eq:impedanceeq} to solutions of the Darboux equation  $\mathbf{L}_{\frac{1}{a}}v=-\lambda v$. For this reason, $\mathbf{D}_a$ is referred to as the {\bf Darboux transform} of Eq. \eqref{eq:impedanceeq}.
	\end{remark}
	
	Fix $x_0\in \overline{I}$ and set
	\[
	\mathbf{J}_av(x):= \int_{x_0}^x\frac{v(t)}{a^2(t)}ds,\quad v\in L_{\frac{1}{a}}^p(I),
	\]
	Since $\frac{dx}{a^2(x)}$ is a finite measure on $I$, the embedding $L_{\frac{1}{a}}^p(I)\hookrightarrow L_{\frac{1}{a}}^1(I)$ is continuous \cite[Prop. 6.13]{folland} and the integral is well defined with $\mathbf{J}_a\in \mathcal{B}(L_{\frac{1}{a}}^p(I), L_a^p(I))$. Furthermore, $\mathbf{J}_a: L_{\frac{1}{a}}^p(I)\rightarrow \mathscr{D}_p(\mathbf{D}_a)$ and is a right-inverse for the operator $\mathbf{D}_a$.

	For $u,v\in \mathscr{D}_p(\mathbf{L}_a)$, the $a$-Wronskian of $u,v$ is defined by
	\[
	W_a[u,v](x)=\begin{vmatrix}
		u(x) & v(x)\\
		a^2(x)u'(x) & a^2(x)v'(x)
	\end{vmatrix}=u(x)\mathbf{D}_av(x)-v(x)\mathbf{D}_au(x).
	\]
	A straightforward calculation shows that if $u$ and $v$ are solutions of \eqref{eq:impedanceeq}, then $W_a[u,v]$ is constant on $I$ and $u$ and $v$ are linearly independent iff their $a$-Wronskian is not zero. In particular, the function $\mathbf{J}_a[1]$ is a solution of Eq. \eqref{eq:impedanceeq} with $\lambda=0$, linearly independent with respect to the constant solution $1$, and is known as the {\it Abel solution}.

	Since for $u\in \mathscr{D}_p(\mathbf{L}_a)$, $\mathbf{D}_au\in C(\overline{I})$, it is natural and convenient to consider initial value problems in terms of $u$ and $\mathbf{D}_au$.
	\begin{theorem}
		Given $\lambda\in \mathbb{C}$, $g\in L_a^p(I)$, $\xi\in \overline{I}$, and $(u_0,u_1)\in \mathbb{C}^2$, the Cauchy problem for Eq. \eqref{eq:impedancegeneral} with  initial conditions
		\[
		u(\xi)=u_0,\quad \mathbf{D}_au(\xi)=u_1,
		\]
		has a unique solution $u\in \mathscr{D}(\mathbf{L}_a)$.
	\end{theorem}
	\begin{proof}
		The result is a particular case of \cite[Th. 4.1.1]{benewitz}.
	\end{proof}
	
	\begin{remark}\label{Remark:nonhomogenouseq}
		Let  $\mathbf{R}_a: L_a^p(I)\rightarrow L_a^p(I)$ be the operator defined by
		\[
		\mathbf{R}_av(x):=\mathbf{J}_a\mathbf{J}_{\frac{1}{a}}v(x)= \int_{x_0}^x\frac{1}{a^2(t)}\left[\int_{x_0}^ta^2(s)v(s)ds\right]dt,\quad v\in L_a^p(I).
		\]
		Then $\mathcal{B}(L^p_a(I))$ and $\mathbf{R}_av\in \mathscr{D}_p(\mathbf{L}_a)$ for all $v\in L_a^p(I)$, with $\mathbf{L}_a\mathbf{R}_av=v$. Furthermore, for $v\in L_a^p(I)$ and $(u_0,u_1)\in \mathbb{C}^2$, the unique solution of the Cauchy problem $\mathbf{L}_au=v$, $u(x_0)=u_0, \mathbf{D}_au(x_0)=u_1$, is given by
		\begin{equation}\label{eq:Solnonhomogenouscauchy}
			u=\mathbf{R}_av+u_1\mathbf{J}_a[1]+u_0.
		\end{equation}
	\end{remark}

	\section{Formal powers, generalized derivatives and Taylor polynomials}

	Let $\rho\in \mathbb{C}$ satisfying $\lambda=\rho^2$, and let $e_a(\rho,x)$ be the unique solution of Eq. \eqref{eq:impedanceeq} satisfying the initial conditions
	\begin{equation}\label{eq:initialconditionssole}
		e_a(\rho,x_0)=1,\quad \mathbf{D}_ae_a(\rho,x_0)=i\rho.
	\end{equation}
	Define the functions
	\begin{equation}\label{eq:definitionsolutioncosineandsine}
		C_a(\rho,x):=\frac{e_a(\rho,x)+e_a(-\rho,x)}{2},\quad S_a(\rho,x):= \frac{e_a(\rho,x)-e_a(-\rho,x)}{2i\rho},\quad x\in I, \; \rho \in \mathbb{C}.
	\end{equation}
	Hence $C_a(\rho,x)$ and $S_a(\rho,x)$ are the solutions of Eq. \eqref{eq:impedanceeq} satisfying the initial conditions
	\begin{align}
		C_a(\rho,x_0)=1,  & \;\; \mathbf{D}_aC_a (\rho,x_0)=0, \label{eq:conditionscosine}\\
		S_a(\rho,x_0)=0, & \;\;\mathbf{D}_aS_a(\rho,x_0)=1. \label{eq:conditionssine}
	\end{align}
	
	The following result is the well-known spectral parameter power series representation (SPPS) for the solution $e(\rho,x)$. The proof can be found in \cite[Th. 7]{sppscampos} and \cite[Th. 1]{spps}.
	
	\begin{theorem}\label{Th:SPPS}
		The solutions $e_a(\rho,x)$ admits the SPPS representation
		\begin{equation}\label{eq:SPPSseries}
			e_a(\rho,x)=\sum_{k=0}^{\infty}\frac{(i\rho)^k\varphi_a^{(k)}(x)}{k!},\qquad x\in I, \; \rho \in \mathbb{C},
		\end{equation}
		where
		\begin{align}
			\varphi_{a}^{(0)}(x)& :=1, \label{eq:1stformalpower}\\
			\varphi_{a}^{(1)}(x) &:= \mathbf{J}_a[1](x),\label{eq:2ndformalpower}\\
			\varphi_{a}^{(k)}(x) & :=k(k-1)\mathbf{R}_a\varphi_a^{(k-1)}(x),\quad k\geq 2. \label{eq:kformalpower}
		\end{align}
		Consequently,
		\begin{equation}
			C_a(\rho,x)=\sum_{k=0}^{\infty}\frac{(-1)^k\rho^{2k}\varphi_a^{(2k)}}{(2k)!},\quad S_a(\rho,x)=\sum_{k=0}^{\infty}\frac{(-1)^k\rho^{2k}\varphi_a^{(2k+1)}}{(2k+1)!},\;\; x\in I, \; \rho \in \mathbb{C}.
		\end{equation}
		The series converges uniformly on $\overline{I}$ with respect to the variable $x$, and the series of the first and second derivatives converges in the norms of $L_{\frac{1}{a}}^p(I)$ and $L_a^p(I)$, respectively. The series converges uniformly and absolutely in the spectral parameter $\rho$ on compact subsets of $\mathbb{C}$.
	\end{theorem}
	\begin{definition}
		The functions $\{\varphi_{a}^{(k)}\}_{k=0}^{\infty}$ are called the {\bf formal powers associated with} $a$ and are said to be {\it centered at} the point $x_0$.
	\end{definition}
	
	\begin{remark}
		From equations \eqref{eq:1stformalpower}-\eqref{eq:kformalpower}, we obtain that
		\begin{equation}\label{eq:Lbasis}
			\mathbf{L}_a\varphi_{a}^{(0)}=\mathbf{L}_a\varphi_{a}^{(1)}=0,\quad \mathbf{L}_a\varphi_{a}^{(k)}=k(k-1)\varphi_{a}^{(k-2)},\;\; k\geq 2.
		\end{equation}
		Following \cite{camposstandard,fage}, we say that $\{\varphi_{a}^{(k)}\}_{k=0}^{\infty}$ is an $\mathbf{L}_a$-basis.
	\end{remark}
	
	The following proposition establishes a relation between formal powers $\{\varphi_a^{(k)}\}_{k=0}^{\infty}$ and $\{\varphi_{\frac{1}{a}}^{(k)}\}_{k=0}$.
	
	\begin{proposition}
		The following relations hold:
		\begin{align}
			\mathbf{D}_a\varphi_a^{(k)}& =k\varphi_{\frac{1}{a}}^{(k-1)} \qquad \forall k\in \mathbb{N},\label{eq:derivativeformalpower}\\
			\mathbf{J}_a\varphi_{\frac{1}{a}}^{(k-1)} & =\frac{\varphi_a^{(k)}}{k}\qquad \forall k\in \mathbb{N}.\label{eq:integralformalpower}
		\end{align}
	\end{proposition}
	\begin{proof}
		First, observe that since $\varphi_a^{(k)}(0)=0$ for all $k\in \mathbb{N}$, we have $\mathbf{J}_a\mathbf{D}_a\varphi_{a}^{(k)}=\varphi_{a}^{(k)}$, so equations \eqref{eq:derivativeformalpower} and \eqref{eq:integralformalpower} are equivalent. We prove \eqref{eq:derivativeformalpower} by strong induction on $k$. For $k=1$ we have
		\[
		\mathbf{D}_a\varphi_a^{(1)}=\mathbf{D}_a\mathbf{J}_a[1]=1=\varphi_{\frac{1}{a}}^{(0)}.
		\]
		Now, assume that  \eqref{eq:derivativeformalpower} is valid for $1\leq j\leq k$, $k>1$. We show that it also holds for \eqref{eq:derivativeformalpower} for $k+1$:
		\begin{align*}
			\mathbf{D}_a\varphi_a^{(k+1)}&= (k+1)k\mathbf{D}_a\mathbf{J}_a\mathbf{J}_{\frac{1}{a}}\varphi_a^{(k-1)} = (k+1)k\mathbf{J}_{\frac{1}{a}}\varphi_a^{(k-1)}.
		\end{align*}
		By the induction hypothesis, $\mathbf{D}_a\varphi_a^{(k-1)}=(k-1)\varphi_{\frac{1}{a}}^{(k-2)}$, and since $k-1>0$, we have the relation $\varphi_a^{(k-1)}=(k-1)\mathbf{J}_a\varphi_{\frac{1}{a}}^{(k-2)}$.
		Thus,
		\begin{align*}
			\mathbf{D}_a\varphi_a^{(k+1)}  &= (k+1)(k)(k-1)\mathbf{J}_{\frac{1}{a}}\mathbf{J}_a\varphi_{\frac{1}{a}}^{(k-2)}= (k+1)\varphi_{\frac{1}{a}}^{(k)}.
		\end{align*}
		This concludes the induction and proves \eqref{eq:derivativeformalpower} for all $k\in \mathbb{N}$, as desired.
	\end{proof}
	
	\begin{remark}
		Using the SPPS series of $e_a(\rho,x)$, we have
		\[
		\mathbf{D}_ae_a(\rho,x)= \sum_{k=0}^{\infty}\frac{(i\rho)^k}{k!}\mathbf{D}_a\varphi_a^{(k)}(x)=\sum_{k=1}^{\infty}\frac{(i\rho)^k}{(k-1)!}\varphi_{\frac{1}{a}}^{(k-1)}(x)=i\rho e_{\frac{1}{a}}(\rho,x).
		\]
		Consequently,
		\begin{align*}
			\mathbf{D}_aC_a(\rho,x) & = \frac{i\rho e_{\frac{1}{a}}(\rho,x)-i\rho e_{\frac{1}{a}}(-\rho,x)}{2}=-\rho^2 S_{\frac{1}{a}}(\rho,x),\\
			\mathbf{D}_aS_a(\rho,x) & = \frac{i\rho e_{\frac{1}{a}}(\rho,x)+i\rho e_{\frac{1}{a}}(-\rho,x)}{2i\rho}=C_{\frac{1}{a}}(\rho,x).
		\end{align*}
	\end{remark}
	\begin{theorem}\label{Prop:convergenceformalpowers}
		Suppose that $\{a_n\}_{n=0}^{\infty}$ is a sequence of proper impedance functions such that
		\begin{equation}\label{eq:conditionconvergenceona}
			a_n^{2j}\rightarrow a^{2j} \; \; \mbox{in } L^1(I) \; \text{ as } n\rightarrow \infty\text{ for } j=-1,1.
		\end{equation}
		Then, for each  $k\in \mathbb{N}$, we have $\varphi_{a_n}^{(k)}\rightarrow \varphi_a^{(k)}$ and $\varphi_{\frac{1}{a_n}}^{(k)}\rightarrow \varphi_{\frac{1}{a}}^{(k)}$ as $n\rightarrow \infty$ uniformly on $\overline{I}$.
		In particular, the conclusion is valid when $a_n\rightarrow a$ uniformly on $\overline{I}$.
	\end{theorem}
	\begin{proof}
		Indeed, the result is trivially true for $k=0$ and
		\begin{align*}
			\|\varphi_{a_n}^{(1)}-\varphi_a^{(1)}\|_{L^{\infty}(I)}\leq \int_I\left|\frac{1}{a^2_n(t)}-\frac{1}{a^2(t)}\right| dt\rightarrow 0.
		\end{align*}
		Suppose that $\varphi_{a_n}^{(j)}\rightarrow \varphi_a^{(j)}$ uniformly for $1\leq j< k$. Hence
		\begin{align*}
			\|\mathbf{J}_{\frac{1}{a_n}}\varphi_{a_n}^{(k-2)}-\mathbf{J}_{\frac{1}{a}}\varphi_a^{(k-2)}\|_{L^{\infty}(I)} \leq & \|\mathbf{J}_{\frac{1}{a_n}}\varphi_{a_n}^{(k-2)}-\mathbf{J}_{\frac{1}{a_n}}\varphi_a^{(k-2)}\|_{L^{\infty}(I)}+\|\mathbf{J}_{\frac{1}{a_n}}\varphi_a^{(k-2)}-\mathbf{J}_{\frac{1}{a}}\varphi_a^{(k-2)}\|_{L^{\infty}(I)}\\
			\leq & \|a_n^2\|_{L^1(I)}\|\varphi_{a_n}^{(k-2)}-\varphi_a^{(k-2)}\|_{L^{\infty}(I)}\\
			& \; +\|a_n^2-a^2\|_{L^1(I)}\|\varphi_a^{(k-2)}\|_{L^{\infty}(I)}.
		\end{align*}
		By \eqref{eq:conditionconvergenceona}, the sequence $\{\|a_n^2\|_{L^1(I)}\}$ is bounded, and the induction hypothesis along with condition \eqref{eq:conditionconvergenceona} implies $\|\mathbf{J}_{\frac{1}{a_n}}\varphi_{a_n}^{(k-2)}-\mathbf{J}_{\frac{1}{a}}\varphi_a^{(k-2)}\|_{L^{\infty}(I)}\rightarrow 0$. Hence $\mathbf{J}_{\frac{1}{a_n}}\varphi_{a_n}^{(k-2)}$ converges uniformly to $\mathbf{J}_{\frac{1}{a}}\varphi_a^{(k-2)}$. Applying the same procedure, we obtain that
		\[
		\varphi_{a_n}^{(k)}=k(k-1)\mathbf{J}_{a_n}\mathbf{J}_{\frac{1}{a_n}}\varphi_{a_n}^{(k-2)} \rightarrow k(k-1)\mathbf{J}_a\mathbf{J}_{\frac{1}{a}}\varphi_a^{(k-2)}=\varphi_a^{(k)}\text{   uniformly on } \overline{I}.
		\]
	\end{proof}

	We denote by $\Pi(\overline{I})$ the set of polynomial functions defined on $\overline{I}$. We extend this concept to the subspace generated by the formal powers.
	
	\begin{definition}
		A {\bf formal polynomial associated with} $a$, is a function of the form
		\[
		\widehat{p}=\sum_{k=0}^{N}\alpha_k\varphi_{a}^{(k)},\quad \alpha_0,\dots, \alpha_N\in \mathbb{C},
		\]
		or equivalently, an element of $\operatorname{Span}\{\varphi_{a}^{(k)}\}_{k=0}^{\infty}$. The set of formal polynomials associated with $a$ is denoted by $\Pi_a(\overline{I})$.
	\end{definition}
	\begin{remark}\label{Remark:operatorIa}
		\begin{itemize}
			\item[(i)] $\mathbf{J}_a\left(\Pi_{\frac{1}{a}}(\overline{I})\right)\subset \Pi_a(\overline{I})$. Indeed, let $\widehat{p}\in \Pi_{\frac{1}{a}}(\overline{I})$. Hence, there exist $\alpha_1,\dots, \alpha_N\in \mathbb{C}$ such that $\widehat{p}=\sum_{k=0}^{N}\alpha_k\varphi_{\frac{1}{a}}^{(k)}$. By equation \eqref{eq:integralformalpower} we have
			\[
			\mathbf{J}_a\widehat{p}=\sum_{k=0}^{N}\alpha_k\mathbf{J}_a\varphi_{\frac{1}{a}}^{(k)}=\sum_{k=0}^N\frac{\alpha_k}{k+1}\varphi_{a}^{(k+1)}\in \Pi_{a}(\overline{I}).
			\]
			\item[(ii)]  The equality cannot hold because  $\mathbf{J}_ay(x_0)=0$ for all $y\in L^p(I)$, and then there is no function such that $\mathbf{J}_ay=1\in \Pi_a(\overline{I})$.
			\item[(ii)] Since $\mathbf{R}_a=\mathbf{J}_a\mathbf{J}_{\frac{1}{a}}$, by the previous point $\mathbf{R}_a\left(\Pi_a(\overline{I}))\right)\subset \Pi_a(\overline{I})$.
		\end{itemize}
	\end{remark}
	
	Abusing notation, we denote $\mathbf{D}_{a^{(-1)}}=\mathbf{D}_{\frac{1}{a}}$ and $\mathbf{J}_{a^{(-1)}}=\mathbf{J}_{\frac{1}{a}}$. Following \cite{KravTremblay1}, we introduce some generalized derivatives and integrals associated with $a$.
	
	\begin{definition}
		The generalized $a$-derivatives are defined recursively by
		\begin{equation}\label{eq:aderivative}
			\mathbf{D}_a^{(0)}:=\mathbf{I},\quad  \mathbf{D}_a^{(m)}=\mathbf{D}_{a^{(-1)^{m-1}}}\mathbf{D}_a^{(m-1)},\; n\in \mathbb{N}.
		\end{equation}
		The $a$-integrals are defined as
		\begin{equation}\label{eq:aintegral}
			\mathbf{J}_a^{(0)}:=\mathbf{I} ,\quad \mathbf{J}_a^{(m)}:= \mathbf{J}_a^{(m-1)}\mathbf{J}_{a^{(-1)^{m}}}, \; m\in \mathbb{N}.
		\end{equation}
		Here, $\mathbf{I}$ denotes the identity operator.
	\end{definition}
	
	\begin{proposition}
		The following statements hold:
		\begin{itemize}
			\item[(i)] \begin{equation}\label{eq:generalizedderivativeformalpowers}
				\mathbf{D}_a^{(m)}\varphi_a^{(n)}=\begin{cases}
					n(n-1)\dots (n-m+1)\varphi_{a^{(-1)^{m}}}^{(n-m)}, & \text{ if }\; n\geq m,\\
					0, & \text{ otherwise}.
				\end{cases}
			\end{equation}
			\item[(ii)] \begin{equation}
				\mathbf{J}_a^{(m)}[1]=\frac{1}{m!}\varphi_a^{(m)}.
			\end{equation}
			\item[(iii)] If $\widehat{p}\in \Pi_a(\overline{I})$, then
			\begin{equation}\label{eq:ataylorpol}
				\widehat{p}=\sum_{k=0}^m\frac{\mathbf{D}_a^{(m)}\widehat{p}(x_0)}{m!}\varphi_a^{(m)}.
			\end{equation}
		\end{itemize}
	\end{proposition}
	\begin{proof}
		Statements (i) and (ii) follow by a standard induction argument. Formula \eqref{eq:ataylorpol} is a direct consequence of statement (i) together with the fact that $\varphi_a^{(j)}(x_0)=0$ if $j>0$.
	\end{proof}

	Using relations \eqref{eq:derivativeformalpower}, an induction argument establishes the following equalities.

	\begin{equation}\label{eq:aDerivativealternativeformula}
		\mathbf{D}_a^{(m)}=\prod_{j=1}^{m}\mathbf{D}_{a^{(-1)^{(j-1)}}} \quad \text{and}\quad \mathbf{J}_a^{(m)}=\prod_{j=0}^{m-1}\mathbf{J}_{a^{(-1)^{(m-j)}}}
	\end{equation}

	\begin{proposition}\label{Prop:taylor}
		For $m\in \mathbb{N}$, the general solution of equation $\mathbf{D}_a^{(m)}u=g$, with $g\in L_{a^{(-1)^{m}}}^p(I)$, is given by
		\begin{equation}\label{eq:generaltaylorformula1}
			u=\mathbf{J}_a^{(m)}g+\sum_{k=0}^{m-1}\frac{\mathbf{D}_a^{(k)}u(x_0)}{k!}\varphi_a^{(k)}.
		\end{equation}
		Moreover, if $a\in W^{m-1,p}(I)$, then every function $u\in W^{m,p}(I)$ admits the generalized Taylor approximation
		\begin{equation}\label{eq:generaltaylorformula2}
			u=\mathbf{J}_a^{(m)}[\mathbf{D}_a^{(m)}u]+\sum_{k=0}^{m-1}\frac{\mathbf{D}_a^{(k)}u(x_0)}{k!}\varphi_a^{(k)}.
		\end{equation}
		If $x_0$ is a Lebesgue point of $\mathbf{D}_a^{(m)}u$,  the remainder term satisfies $\mathbf{J}_a^{(m)}[\mathbf{D}_a^{(m)}u] =O((x-x_0)^m)$ as $x\rightarrow x_0$.
	\end{proposition}
	
	\begin{proof}
		From equalities \eqref{eq:aDerivativealternativeformula}, it follows that for any $g\in L_{a^{(-1)^{m}}}^p(\overline{I})$, the function $\mathbf{J}_a^{(m)}g$ belongs to the domain of $\mathbf{D}_a^{(m)}$ and $\mathbf{D}_a^{(m)}\mathbf{J}_a^{(m)}g=g$. First, we prove by induction that the unique solution of the homogeneous equation $\mathbf{D}_a^{(m)}v=0$ has the form $v=\sum_{k=0}^{m-1}\frac{\mathbf{D}_a^{(k)}u(x_0)}{k!}\varphi_a^{(k)}$. For $m=1$, this is immediate. Assume the formula holds for $1\leq k\leq m$, with $m\geq 2$. Suppose that $\mathbf{D}_a^{(m+1)}v=0$. Hence $(a^2)^{(-1)^{m}}D\mathbf{D}_a^{(m)}v=0$, which implies that $\mathbf{D}_a^{(m)}v=c$, with $c\in \mathbb{C}$. A particular solution is given by $\mathbf{J}_a^{(m)}[c]=\frac{c}{m!}\varphi_a^{(m)}$. Then $w=v-\mathbf{J}_a^{(m)}[c]$ is a solution of $\mathbf{D}_a^{(m)}v=0$, and by the induction hypothesis, $w=\sum_{k=0}^{m-1}\frac{\mathbf{D}_a^{(k)}u(x_0)}{k!}\varphi_a^{(k)}$.  Since $\mathbf{D}^{(m)}\mathbf{J}_a^{(m)}[c](x_0)=c$, we conclude that $v=\sum_{k=0}^{m}\frac{\mathbf{D}_a^{(k)}u(x_0)}{k!}\varphi_a^{(k)}$.
		
		For the non-homogeneous equation, consider  $v=u-\mathbf{J}_a^{(m)}g$, which is a solution of $\mathbf{D}_a^{(m)}v=0$. Thus, $v=\sum_{k=0}^{m-1}\frac{\mathbf{D}_a^{(k)}u(x_0)}{k!}\varphi_a^{(k)}$. Since $\mathbf{D}_a^{(j)}v(x_0)=\mathbf{D}_a^{(j)}u(x_0)-\mathbf{J}_a^{(m-j)}g(x_0)=\mathbf{D}_a^{(j)}u(x_0)$ for $j<m$, we conclude \eqref{eq:generaltaylorformula1}, as desired.
		
		Finally, if $a\in W^{m-1,p}(I)$ and $u\in W^{m,p}(I)$, then $\mathbf{D}_a^{(m)}u\in L_{a^{(-1)^{m}}}^p(I)$ and formula \eqref{eq:generaltaylorformula2} follows directly from \eqref{eq:generaltaylorformula1}. When $x_0$ is a Lebesgue point of $\mathbf{D}_a^{(m)}u$,  the Lesbesgue differentiation theorem \cite[Ch. 3., Th. 3. 21]{folland} together with the L'Hospital rule yield
		\begin{align*}
			\lim_{x\rightarrow x_0}\frac{1}{(x-x_0)^m}\mathbf{J}_a^{(m)}\mathbf{D}_a^{(m)}u(x)&=\frac{1}{m!}\prod_{k=2}^{m-1}{a^2(x_0)}^{(-1)^{m-k}}\lim_{x\rightarrow 0}\frac{1}{x-x_0}\int_{x_0}^x(a^2(t))^{(-1)^{m-1}}\mathbf{D}_a^{(m)}u(t)dt\\
			&=\frac{1}{m!}\prod_{k=1}^{m-1}{a^2(x_0)}^{(-1)^{m-k}}\mathbf{D}_a^{(m)}u(x_0).
		\end{align*}
		
		Thus, $\mathbf{J}_a^{(m)}\mathbf{D}_a^{(m)}u(x)=O((x-x_0)^m)$ as $x\rightarrow x_0$.
	\end{proof}
	
	\begin{remark}\label{Remark:Jam}
		Let $m,k\in \mathbb{N}$. By \eqref{eq:generalizedderivativeformalpowers}, $\varphi_{a^{(-1)^m}}^{(k)}=\frac{1}{(m+k)\dots (k+1)}\mathbf{D}_a^{(m)}\varphi_a^{(m+k)}$. Proposition \ref{Prop:taylor} implies that
		\begin{equation*}
			\mathbf{J}_a^{(m)}\varphi_{a^{(-1)^m}}^{(k)}=\frac{1}{(m+k)\dots (k+1)}\mathbf{J}_a^{(m)}\mathbf{D}_a^{(m)}\varphi_a^{(m+k)}=\frac{\varphi_a^{(m+k)}}{(m+k)\dots (k+1)}+\widehat{p},
		\end{equation*}
		where $\widehat{p}\in \Pi_a(\overline{I})$ is of a degree less than $m+k$. Consequently, $\mathbf{J}_a^{(m)}\varphi_{a^{(-1)^m}}^{(k)}\in \Pi_a(\overline{I})$ for all $k\in \mathbb{N}$. $\therefore \mathbf{J}_a^{(m)}\left(\Pi_{a^{(-1)^m}}(\overline{I})\right)\subset \Pi_a(\overline{I})$.
	\end{remark}
	
	\section{Completeness of the formal powers}\label{Sec: Completeness}
	Recall that a subset $M$ of a Banach space $X$ is called {\it complete} in $X$ if  $\operatorname{Span}M$ is dense in $X$. In the case of a sequence of vectors $\{u_k\}_{k=0}^{\infty}$, if the set is complete we say that it forms a {\it complete system}. An element $u\in \operatorname{Span}\{u_k\}_{k=0}^{\infty}$  can be written as $u=\sum_{k=0}^{N}\alpha_ku_k$, with $\alpha_0,\dots,\alpha_N\in \mathbb{C}$. For example, by the Weierstrass approximation theorem and the density of $C(\overline{I})$ in $L^p(I)$, is known that $\Pi(\overline{I})$ is complete in $L^p(I)$, for $1\leq p<\infty$.

	To prove the completeness of the formal powers $\{\varphi_a^{(k)}\}_{k=0}^{\infty}$ in $L^p(I)$, $1\leq p<\infty$, we recall some results related to Sturm-Liouville problems associated to Eq. \eqref{eq:impedanceeq}.
	
	First, suppose that $I=(\ell_1,\ell_2)$ and consider the operator $\mathbf{L}_a^D:\mathscr{D}_p(\mathbf{L}_a^D)\subset L_a^p(I) \rightarrow L_a^p(I)$ given by the restriction of $\mathbf{L}_a$ to the subspace $\mathscr{D}_p(\mathbf{L}_a^D):=\mathscr{D}_p(\mathbf{L}_a)\cap W^{1,1}_0(I)$ (that is, the operator subject to Dirichlet-to-Dirichlet boundary conditions $u(\ell_1)=u(\ell_2)=0$).
	
	\begin{lemma}\label{lemma:invertibledirichletoperator}
		Let $1\leq p<\infty$. The operator $\mathbf{R}_a^D:L_a^p(I)\rightarrow L_a^p(I)$ given by
		\begin{equation}\label{eq:rightinversedirichlet}
			\mathbf{R}_a^Dg= \mathbf{R}_ag+\frac{\mathbf{R}_ag(\ell_1)-\mathbf{R}_ag(\ell_2)}{\|a\|_{L^2(I)}}\varphi_a^{(1)}+\frac{\mathbf{R}_ag(\ell_1)\varphi_a^{(1)}(\ell_2)-\mathbf{R}_ag(\ell_2)\varphi_a^{(1)}(\ell_1)}{\|a\|_{L^2(I)}}\cdot 1,
		\end{equation}
		is a right inverse for $\mathbf{L}_a^D$.
	\end{lemma}
	\begin{proof}
		Given $g\in L_a^p(I)$, consider the function
		\[
		G=\mathbf{R}_ag+c_1\varphi_a^{(1)}+c_0
		\]
		which satisfies $G\in \mathscr{D}_p(\mathbf{L}_a)$ and $\mathbf{L}_aG=g$. Now, we look for $c_0,c_1\in \mathbb{C}$ such that $G(\ell_1)=G(\ell_2)=0$. This leads to the system of equations
		\begin{align*}
			-\mathbf{R}_ag(\ell_1) &= c_0+c_1\varphi_a^{(1)}(\ell_1),\\
			-\mathbf{R}_ag(\ell_2) &= c_0+c_1\varphi_a^{(1)}(\ell_2)
		\end{align*}
		The coefficient matrix has determinant
		\begin{align*}
			\begin{vmatrix}
				1 & \varphi_a^{(1)}(\ell_1),\\
				1 & \varphi_a^{(1)}(\ell_2)
			\end{vmatrix} &= \varphi_a^{(1)}(\ell_2)-\varphi_a^{(1)}(\ell_1)= \int_{x_0}^{\ell_2}\frac{dt}{a^2(t)}-\int_{x_0}^{\ell_1}\frac{dt}{a^2(t)}= \int_{\ell_1}^{\ell_2}\frac{dt}{a^2(t)}>0.
		\end{align*}
		Since $a$ is positive, then the system has a unique solution and a straightforward computation shows that these values of $c_0,c_1$ yield exactly the formula for $G=\mathbf{R}_a^Dg$ given in \eqref{eq:rightinversedirichlet}. Since
		\[
		\|\mathbf{R}_ag\|_{L^{\infty}(I)}\leq \|a^{-1}\|_{L^2(I)}\|a\|_{L^2(I)}^{\frac{p-1}{p}}\|g\|_{L_a^p(I)},
		\]
		we conclude that $\mathbf{R}_a^D\in \mathcal{B}(L_a^p(I))$.
	\end{proof}

	\begin{lemma}\label{lemma:denseoperator}
		For $1<p<\infty$, the operator $\mathbf{L}_a^D$ is densely defined on $L_a^p(I)$.
	\end{lemma}
	
	\begin{proof}
		Let $v\in L_a^{p'}(I)$, with $p'=\frac{p}{p-1}$. Suppose that
		\begin{equation}\label{eq:auxiliar8}
			\int_{I}a^2(x)u(x)v(x)dx=0\qquad \forall u\in \mathscr{D}(\mathbf{L}_a^D).
		\end{equation}
		Let $\phi\in C_0^{\infty}(I)$ be arbitrary. By Lemma \ref{lemma:invertibledirichletoperator}, $\Phi=\mathbf{R}_a^D\phi \in \mathscr{D}_p(\mathbf{L}_a^D)$ is such that $\mathbf{L}_a\Phi=\phi$. Similarly, $V=\mathbf{R}_a^Dv\in \mathscr{D}_{p'}(\mathbf{L}_a^D)$ and $\mathbf{L}_aV=v$. Integration by parts and the boundary conditions in \eqref{eq:auxiliar1} yield
		\begin{equation}
			0=\int_{I}a^2\Phi v=\int_Ia^2\Phi (\mathbf{L}_aV)=\int_{I}a^2(\mathbf{L}_a\Phi) V=\int_{I}a^2\phi V.
		\end{equation}
		Since this relation holds for all $\phi\in C_0^{\infty}(I)$, we obtain that $V=0$ a.e. in $I$, and consequently $v=\mathbf{L}_aV=0$. Since $L_a^{p'}(I)$ is isometrically isomorphic to the dual of $L_a^p(I)$, we conclude that $\mathscr{D}(\mathbf{L}_a^D)$ is dense in $L_a^p(I)$. 
	\end{proof}
	
	\begin{lemma}\label{lemma:densesubspace2}
		For $1<p<\infty$, the subspace $\mathbf{R}_a^D(C_0(I))$ is dense in $\mathscr{D}_p(\mathbf{L}_a^D)$. Consequently, $\mathbf{R}_a^D(C_0^{\infty}(I))$ is dense in $L_a^p(I)$.
	\end{lemma}
	\begin{proof}
		Let $g\in \mathscr{D}_p(\mathbf{L}_a)$ and $\varepsilon>0$.  Since $a^2(x)dx$ is a finite Radon measure on $I$, it follows that $C_0(I)$ is dense in $L_a^p(I)$ \cite[Ch. 7, Prop. 7.9]{folland}. Because $\mathbf{L}_ag\in L_a^p(I)$, there exists $\psi\in C_0(I)$ such that $\|\mathbf{L}_ag-\psi\|_{L_a^p(I)}<\frac{\varepsilon}{M}$, where $M$ is the operator-norm of $\mathbf{R}_a^D$. Since $g=\mathbf{R}_a^D\mathbf{L}_ag$, we have
		\[
		\|g-\mathbf{R}_a^D\psi\|_{L_a^p(I)}\leq M \|\mathbf{L}_ag-\psi\|_{L_a^p(I
			)}<\varepsilon.
		\]
		Thus, $\mathbf{R}_a^D(C_0(I))$ is dense in $\mathscr{D}_p(\mathbf{L}_a^D)$. The density of $\mathbf{R}_a^D(C_0(I))$ in $L_a^p(I)$ follows from Lemma \ref{lemma:denseoperator}. Since $C_0^{\infty}(I)$ is dense in $C_0(I)$ with the uniform norm and the embedding $C_0(I)\hookrightarrow L_a^p(I)$ is continuous, it follows that  $\mathbf{R}_a^D(C_0^{\infty}(I))$ is dense in $L_a^p(I)$.
	\end{proof}

	We recall the following result for the Sturm-Liouville problem with Dirichlet-to-Dirichlet condition on $L_a^2(I)$.
	
	\begin{theorem}\label{Th:Egeinexpasion}
		The Sturm-Liouville problem consisting of Eq. \eqref{eq:impedanceeq} with the Dirichlet-to-Dirichlet conditions
		\[
		u(\ell_1)=u(\ell_2)=0,
		\]
		admits a countable set of simple eigenvalues $\{\lambda_n\}_{n=1}^{\infty}$
		with corresponding normalizing (with respect to the norm of $L_a^2(I)$) eigenfunctions $\{\phi_n\}_{n=0}^{\infty}$, which form orthonormal basis for $L_a^2(I)$.
		
		Furthermore, if $u\in \mathscr{D}_2(\mathbf{L}_a^D)$ , then the Fourier series expansion
		\[
		u(x)=\sum_{n=1}^{\infty}\tilde{u}_n\phi_n(x),\;\;\; \text{where}\quad  \tilde{u}_n:=\langle u,\phi_n\rangle_{L_a^2(I)},\; n\in \mathbb{N},
		\]
		converges uniformly on $\overline{I}$.
	\end{theorem}
	\begin{proof}
		Since $a^2,\frac{1}{a^2}\in L_1(I)$, the problem with Dirichlet-to-Dirichlet conditions is regular, so the first part of the result follows from \cite[Th. 2.7.4]{miklavic}. The expansion theorem is a consequence of \cite[Cor. 4.3.3]{benewitz}.
	\end{proof}
	\begin{theorem}\label{Th:CompletenessformalLp}
		For $1\leq p<\infty$, the formal powers $\{\varphi_a^{(k)}\}_{k=0}^{\infty}$ are a complete system in $L_a^p(I)$.
	\end{theorem}
	\begin{proof}
		First, consider $p>1$. Let $f\in L_a^p(I)$ and $\varepsilon>0$. By Lemma \ref{lemma:densesubspace2}, there exists $\psi\in C_0(I)$ with $\Psi=\mathbf{R}_a^D\psi\in \mathscr{D}_p(\mathbf{L}_a^D)$ such that
		
		\begin{equation}\label{eq:auxiliar0}
			\|f-\Psi\|_{L_a^p(I)}<\frac{\varepsilon}{2}.
		\end{equation}
		Note that $\Psi \in C(\overline{I})\subset L_a^2(I)$ with $\mathbf{L}_a\Psi =\psi \in C_0(I)\subset L_a^2(I)$, so $\Psi\in \mathscr{D}_2(\mathbf{L}_a^D)$. Thus, by the second part of Theorem \ref{Th:Egeinexpasion}, there exist $\phi_1,\dots, \phi_N$ eigenfunctions of the Sturm-Liouville problem of Eq. \eqref{eq:impedanceeq} with the Dirichlet-to-Dirichlet conditions such that
		\begin{equation}\label{eq:auxliar1}
			\left\|\Psi-\sum_{n=1}^{N}\widetilde{\Psi}_n\phi_n\right\|_{L^{\infty}(I)}<\frac{\varepsilon}{4\|a\|_{L^2(I)}},\quad \text{where }\; \widetilde{\Psi}_n=\langle \Psi,\phi_n\rangle_{L_a^2(I)}, \; n=1,\dots, N.
		\end{equation}
		Since for each eigenvalue $\lambda_n=\rho_n^2$,  the pair $\{e_a(\rho_n,x),e_a(-\rho_n,x)\}$ forms a fundamental set of solutions for Eq. \eqref{eq:impedanceeq}, every eigenfunction $\phi_n$  is a linear combination of $e_a(\rho_n,x)$ and $e_a(-\rho_n,x)$, an according to Theorem \ref{Th:SPPS} they can be expanding as a uniformly convergent series of formal powers $\{\varphi_a^{(k)}\}_{k=0}^{\infty}$. Consequently, for each $n=1,\dots, N$, there exist constants $\alpha_{0,1},\dots, \alpha_{M_N,N}\in \mathbb{C}$ such that
		\begin{equation}\label{eq:auxiliar3}
			\left\| \phi_n-\sum_{k=0}^{M_N}\alpha_{k,n}\varphi_a^{(k)}\right\|_{L^{\infty}(I)}<\frac{\varepsilon}{4\|a\|_{L^2(I)}CN},
		\end{equation}
		where $C=\max\{1,|\widetilde{\Psi}_1|,\dots, |\widetilde{\Psi}_n|\}$. Set $\widehat{p}=\sum_{n=1}^{N}\sum_{k=0}^{M_n}\widetilde{\Psi}_n\alpha_{k,n}\varphi_a^{(k)}\in \Pi_a(\overline{I})$. From inequalities \eqref{eq:auxliar1} and \eqref{eq:auxiliar3}, we obtain:
		\begin{align*}
			\|\Psi-\widehat{p}\|_{L^{\infty}(I)} & \leq \left\|\Psi-\sum_{n=1}^N\widetilde{\Psi}_n\phi_n\right\|_{L^{\infty}(I)}+\sum_{n=1}^N|\widetilde{\Psi}_n|\left\|\phi_n-\sum_{k=0}^{M_n}\alpha_{k,n}\varphi_a^{(k)}\right\|_{L^{\infty}(I)} <\frac{\varepsilon}{2\|a\|_{L^2(I)}}.
		\end{align*}
		Thus,
		\begin{align*}
			\|f-\widehat{p}\|_{L_a^p(I)}\leq\|f-\Psi\|_{L_a^p(I)}+\|\Psi-\widehat{p}\|_{L_a^p(I)}<\varepsilon.
		\end{align*}
		Therefore, $\Pi_a(\overline{I})$ is dense in $L_a^p(I)$. For $p=1$, note that the embedding $L_a^2(I)\hookrightarrow L_a^1(I)$ is continuous and  since $C_0(I)\subset L_a^2(I)\subset L_a^1(I)$, then $L_a^2(I)$ is dense in $L_a^1(I)$. Therefore, $\Pi_a(\overline{I})$ is dense in $L_a^1(I)$.
	\end{proof}

	\begin{theorem}\label{Th:Completenessformalpowerssobolev}
		Let $1\leq p<\infty$.
		\begin{itemize}
			\item[(i)] If $a,\frac{1}{a}\in L^{\infty}(I)$, then the formal powers $\{\varphi_a^{(k)}\}_{k=0}^{\infty}$ are a complete system in $W^{1,p}(I)$.
			\item[(ii)] For $m\in \mathbb{N}$, $m\geq 2$, if $a\in W^{m-1,p}(I)$, then the formal powers $\{\varphi_a^{(k)}\}_{k=0}^{\infty}$ are a complete system in $W^{m,p}(I)$.
		\end{itemize}
		
	\end{theorem}
	\begin{proof}
		\begin{itemize}
			\item[(i)] Let $u\in W^{1,p}(I)$. First, note that $u\in \mathscr{D}_p(\mathbf{D}_a)$ and 
			\begin{equation}\label{eq:auxiliar4}
				u=\mathbf{I}_a\left[\mathbf{D}_au\right]+u(x_0).
			\end{equation}
			By Theorem \ref{Th:CompletenessformalLp}, there exists a sequence $\{\widehat{p}_n\}_{n=0}^{\infty}\subset \Pi_{\frac{1}{a}}(\overline{I})$ such that $\widehat{p}_n\rightarrow \mathbf{D}_au$ in $L^p(I)$. By Remark \ref{Remark:operatorIa}, $\mathbf{I}_a\widehat{p}_n\in \Pi_a(\overline{I})$ for all $n\in \mathbb{N}_0$ and $\mathbf{I}_a\widehat{p}_n\rightarrow \mathbf{I}_a[\mathbf{D}_au]$ in $W^{1,p}(I)$. Then the sequence
			\[
			\widehat{q}_n=\mathbf{I}_a\widehat{p}_n+u(x_0)\in \Pi_a(\overline{I}),\quad n\in \mathbb{N}_0,
			\]
			satisfies that $\widehat{q}_n\rightarrow u$ in $W^{1,p}(I)$. Consequently, $\Pi_a(\overline{I})$ is dense in $W^{1,p}(I)$.
			
			\item[(ii)] Given $u\in W^{m,p}(I)$, by Proposition \eqref{Prop:taylor}, $u$ can be written as
			\[
			u=\mathbf{J}_a^{(m)}[\mathbf{D}_a^{(m)}u]+\sum_{k=0}^{m}\frac{\mathbf{D}_a^{(k)}u(x_0)}{k!}\varphi_a^{(k)}.
			\]
			Theorem \ref{Th:CompletenessformalLp} implies that there exists a sequence $\{\widehat{p}_n\}_{n=0}^{\infty}\subset \Pi_{a^{(-1)^m}}(\overline{I})$ with $\widehat{p}_n\rightarrow \mathbf{D}_a^{(m)}u$ in $L^p(I)$. Consider
			\[
			\widehat{q}_n:=\mathbf{J}_a^{(m)}[\widehat{p}_n]+\sum_{k=0}^{m-1}\frac{\mathbf{D}_a^{(k)}u(x_0)}{k!}\varphi_a^{(k)},\;\; n\in \mathbb{N}_0.
			\]
			
			By Remark \ref{Remark:Jam}, $\{\widehat{q}_n\}_{n=0}^{\infty}\subset \Pi_a(\overline{I})$ and since $\mathbf{J}_a^{(m)}\in \mathcal{B}(L^p(I), W^{m,p}(I))$, then $\widehat{q}_n\rightarrow u$ in $W^{m,p}(I)$. $\therefore$ $\Pi_a(\overline{I})$ is complete in $W^{m,p}(I)$, as desired.
			
		\end{itemize}
		
	\end{proof}
	\begin{corollary}\label{coro:polynomialaprox}
		Let $m\in \mathbb{N}$ and $1\leq p <\infty$. For any $u\in W^{m,p}(I)$, there exists a sequence $\{p_n\}_{n=0}^{\infty}\subset \Pi(\overline{I})$ such that
		\begin{equation*}\label{eq:polinomialaprox}
			p_n^{(k)}\rightarrow u^{(k)} \text{ uniformly on } I, k=0,\dots, m-1, \;\; \text{and }\;\; p^{(m)}\rightarrow u^{(m)} \text{ in } L^p(I).
		\end{equation*}
		In particular, $\Pi(\overline{I})$ is complete in $W^{m,p}(I)$.
	\end{corollary}
	\begin{proof}
		For $a\equiv 1$, the formal powers coincide with the usual ones. Thus, the result follows from Theorem \ref{Th:Completenessformalpowerssobolev} and the fact that the embedding $W^{m,p}(I)\hookrightarrow C^{m-1}(\overline{I})$ is continuous \cite[p. 217]{brezis}.
	\end{proof}

	\section{Transmutation operators for $a\in W^{1,\infty}(-\ell,\ell)$}
	
	\subsection{Transmutation property}
	
	In this section, we assume that $I=(-\ell,\ell)$, with  $\ell>0$, that $x_0=0$, and  that $a\in W^{1,\infty}(-\ell,\ell)
	$ satisfies the normalization condition that $a(0)=1$. In this case, the operators $\mathbf{I}_a,\mathbf{R}_a, \mathbf{L}_a$ and $\mathbf{D}_a,$ are well defined on the spaces $L^p(-\ell,\ell)$, $W^{1,p}(-\ell,\ell)$ and $W^{2,p}(-\ell,\ell)$ for $1\leq p\leq \infty$.
	
	It is known that for $\frac{a'}{a}\in L^{\infty}(-\ell,\ell)$, the solution $e_a(\rho,x)$ admits the integral representation
	\begin{equation}\label{eq:integralsole}
		e_a(\rho,x)=e^{i\rho x}-i\rho\int_{-x}^{x}K_a(x,t)e^{i\rho t}dt,\qquad -\ell\leq x\leq \ell,
	\end{equation}
	where $K_a(x,t)$ is the Fourier transform of $\frac{e^{i\rho x}-e_a(\rho,x)}{i\rho}$ (see \cite{carroll} and \cite[Sec. 2.1]{mineimpedance1}). Furthermore, the kernel $K_a(x,t)$ satisfies some regularity conditions.

	\begin{theorem}\label{Th:integralrepresentation}
		The integral kernel  $K_a\in W^{1,\infty}(\Omega)$,  where $\Omega=(-\ell,\ell)\times (-\ell,\ell)$, and  satisfies the Goursat conditions
		\begin{equation}\label{eq:Goursatcondition}
			K_a(x,x)=1-\frac{1}{a(x)},\quad K_a(x,-x)=0,\qquad -\ell\leq x\leq \ell.
		\end{equation}
	
\end{theorem}
\begin{proof}
	See \cite{mineimpedance1}, Subsection 2.6, Remark 8.
\end{proof}

Set
\begin{equation}\label{eq:transmutationoperator}
	\mathbf{T}_au(x):= u(x)-\int_{-x}^{x}K_a(x,t)u'(t)dt, \quad u\in W^{1,p}(-\ell,\ell), \; -\ell\leq x\leq \ell.
\end{equation}
The operator is well defined for $u\in W^{1,p}(-\ell,\ell)$, $1\leq p\leq \infty$.

Hence, integral representation \eqref{eq:integralsole} can be expressed as
\begin{equation}\label{eq:transmutationexponential}
	e_a(\rho,x)=\mathbf{T}_a[e^{i\rho x}].
\end{equation}
Consequently, 	\begin{equation}\label{eq:transmutationcosinesine}
	C_a(\rho,x)=\mathbf{T}_a[\cos(\rho x)],\quad S_a(\rho,x)=\mathbf{T}_a\left[\frac{\sin(\rho x)}{\rho}\right].
\end{equation}
We establish the first analytical properties of the operator $\mathbf{T}_a$.
\begin{proposition}\label{Prop:boundedoperatorT}
	For $1\leq p\leq \infty$, the operator $\mathbf{T}_a$ belongs to $ \mathcal{B}(W^{1,p}(-\ell,\ell))$, and its derivative is given by
	\begin{equation}\label{eq:derivativetransmutation}
		D\mathbf{T}_au(x)=\frac{u'(x)}{a(x)}-\int_{-x}^{x}\partial_xK_a(x,t)u'(t)dt.
	\end{equation}
	In the case when $u\in W^{2,p}(-\ell,\ell)$, if $v=\mathbf{T}_au$, then $v(0)=u(0)$ and $v'(0)=u'(0)$, that is,  $\mathbf{T}_a$ preserves the initial conditions at $x=0$.
\end{proposition}
\begin{proof}
	Consider the operator $\mathbf{K}_1u(x):=\int_{-x}^{x}K_a(x,t)u(t)dt$, $u\in W^{1,p}(-\ell,\ell)$. Since $\mathbf{K}_1$ is a Volterra integral operator with kernel $K\in L^{\infty}(\Omega)$, hence $\mathbf{K}_1\in \mathcal{B}(L^p(-\ell,\ell))$. For any $u\in W^{1,p}(-\ell,\ell)$, we have
	\[
	\|\mathbf{T}_au\|_{L^p(-\ell,\ell)}=\|u-\mathbf{K}_1u'\|_{L^p(-\ell,\ell)}\leq \|u\|_{L^p(-\ell,\ell)}+\|\mathbf{K}_1\|_{\mathcal{B}(L^p(-\ell,\ell))}\|u'\|_{L^p(-\ell,\ell)}.
	\]
	For the derivative, using the fact that $K_a\in W^{1,\infty}(\Omega)$ and applying the Leibniz rule, we obtain
	\begin{align*}
		D\mathbf{T}_au(x)& = u'(x)-K_a(x,x)u'(x)-K_a(x,-x)u'(-x)-\int_{-x}^{x}\partial_xK_a(x,t)u'(t)dt\\
		&= u'(x)-\left(1-\frac{1}{a(x)}\right)u'(x)-\int_{-x}^{x}\partial_xK_a(x,t)u'(t)dt\\
		&= \frac{u'(x)}{a(x)}-\int_{-x}^{x}\partial_xK_a(x,t)u'(t)dt.
	\end{align*}
	The second equality is due to the Goursat conditions \eqref{eq:Goursatcondition}. This establishes \eqref{eq:derivativetransmutation}. Now, consider the Volterra operator $\mathbf{K}_2v(x)=\int_{-x}^{x}\partial_xK_a(x,t)v(t)dt$. Since $\partial_xK_a\in L^{\infty}(\Omega)$, it follows that $\mathbf{K}_2\in \mathcal{B}(L^p(-\ell,\ell))$. Thus, 
	\begin{align*}
		\left\|D\mathbf{T}_au\right\|_{L^p(-\ell,\ell)} & = \left\| \frac{u'}{a}-\mathbf{K}_2u'\right\|_{L^p(-\ell,\ell)}\\
		&\leq \left\|a^{-1}\right\|_{L^{\infty}(-\ell,\ell)}\|u\|_{L^p(-\ell,\ell)}+\|\mathbf{K}_2\|_{\mathcal{B}(L^p(-\ell,\ell))}\|u'\|_{L^p(-\ell,\ell)}.
	\end{align*}
	Consequently, for $1\leq p <\infty$,
	\[
	\|\mathbf{T}_au\|_{W^{1,p}(-\ell,\ell)}^p \leq \left[\left(1+\|\mathbf{K}_1\|_{\mathcal{B}(L^p(-\ell,\ell))}\right)^p+\left(\left\|a^{-1}\right\|_{L^{\infty}(-\ell,\ell)}+\|\mathbf{K}_2\|_{\mathcal{B}(L^p(-\ell,\ell))}\right)^p \right] \|u\|_{W^{1,p}(-\ell,\ell)},
	\]
	which implies that $\mathbf{T}_a\in \mathcal{B}(W^{1,p}(-\ell,\ell))$. Similarly for $p=\infty$.
	
	Finally, if $v=\mathbf{T}_au$, $u\in W^{2,p}(-\ell,\ell)$, hence $ v(0)=u(0)$, and by \eqref{eq:derivativetransmutation}, $v'(0)=u'(0)$.
\end{proof}

The following result generalizes the corresponding statement established for $a\in C^1[-\ell,\ell]$ in \cite{mineimpedance1}.

\begin{proposition}\label{Prop:mappingproperty}
	The following relations hold:
	\begin{equation}\label{eq:mappingproperty}
		\mathbf{T}_a[x^k]=\varphi_{a}^{(k)},\quad \forall k\in \mathbb{N}_0.
	\end{equation}
\end{proposition}
\begin{proof}
	Since the Taylor series of $e^{i\rho x}$ converges in $W^{1,\infty}(-\ell,\ell)$, by the integral representation \eqref{eq:transmutationexponential} and the continuity of the operator $\mathbf{T}_a$, we have
	\[
	e(\rho,x)=\mathbf{T}_a\left[\sum_{k=0}^{\infty}\frac{(i\rho)^kx^k}{k!}\right]=\sum_{k=0}^{\infty}\frac{(i\rho)^k\mathbf{T}_a[x^k]}{k!}.
	\]
	Comparing this series with the SPPS series \eqref{eq:SPPSseries}, as Taylor series in $\rho$, we obtain \eqref{eq:mappingproperty}.
\end{proof}

\begin{theorem}\label{Th:transmutationop}
	For $1\leq p <\infty$, the operator $\mathbf{T}_a$ satisfies the following transmutation property:
	\begin{equation}\label{eq:transmutationproperty}
		\mathbf{L}_a\mathbf{T}_au=\mathbf{T}_aD^2u,\quad \forall u\in W^{3,p}(-\ell,\ell).
	\end{equation}
\end{theorem}
\begin{proof}
	We begin by proving the relation \eqref{eq:transmutationoperator} for $p\in \Pi[-\ell,\ell]$. Let $p(x)=\sum_{k=0}^{N}a_kx^k$. By relations \eqref{eq:Lbasis} and the mapping property \eqref{eq:mappingproperty}, we have
	\begin{align*}
		\mathbf{L}_a\mathbf{T}_ap(x) &= \mathbf{L}_a\left(\sum_{k=0}^{N}a_k\mathbf{T}_a[x^k]\right)= \sum_{k=0}^{N}\alpha_k\mathbf{L}_a\varphi_{a}^{(k)}(x)\\
		&= \sum_{k=2}^{N}k(k-1)\alpha_k\varphi_{a}^{(k-2)}= \sum_{k=2}^{N}k(k-1)\alpha_k\mathbf{T}_a[x^{k-2}]= \mathbf{T}_aD^2p(x).
	\end{align*}
	For a general $u\in W^{3,p}(-\ell,\ell)$,  let $v= \mathbf{T}_au$, and note that  \eqref{eq:transmutationproperty} is equivalent to
	\[
	\mathbf{L}_av= \mathbf{T}_aD^2u.
	\]
	By Remark \eqref{Remark:nonhomogenouseq} and since $a(0)=1$, this is equivalent to
	\begin{equation}\label{eq:auxiliar1}
		v= \mathbf{R}_a\mathbf{T}_aD^2u+v(0)\varphi_{a}^{(0)}+v'(0)\varphi_{a}^{(1)}.
	\end{equation}
	We prove relation \eqref{eq:auxiliar1}. Let $\{p_n\}_{n=0}^{\infty}\subset \Pi[-\ell,\ell]$ such that $p_n^{(j)}\rightarrow u^{(j)}$ uniformly on $[-\ell,\ell]$ for $j=0,1,2,$ and $p_n'''\rightarrow u'''$ in $L^p(-\ell,\ell)$.  Since \eqref{eq:transmutationproperty} is valid for each $p_n$, hence
	\[
	\mathbf{T}_ap_n= \mathbf{R}_a\mathbf{T}_aD^2p_n+p_n(0)\varphi_{a}^{(0)}+p_n'(0)\varphi_{a}^{(1)}
	\]
	The continuity of $\mathbf{T}_a$ in $W^{1,p}(-\ell,\ell)$ and $\mathbf{R}_a$ in $L^p(-\ell,\ell)$ yields
	\begin{align*}
		v &=  \lim_{n\rightarrow \infty} \mathbf{T}_ap_n \\
		&= \lim_{n\rightarrow \infty}\left(\mathbf{R}_a\mathbf{T}_aD^2p_n+p_n(0)\varphi_{a}^{(0)}+p_n'(0)\varphi_{a}^{(1)}\right) \\
		&=\mathbf{R}_a\mathbf{T}_aD^2u+u(0)\varphi_{a}^{(0)}+u'(0)\varphi_{a}^{(1)}\\
		&= \mathbf{R}_a\mathbf{T}_aD^2u+v(0)\varphi_{a}^{(0)}+v'(0)\varphi_{a}^{(1)},
	\end{align*}
	because $v(0)=u(0)$ and $v'(0)=u'(0)$. Therefore, the transmutation property\eqref{eq:transmutationproperty} is valid for all $u\in W^{3,p}(-\ell,\ell)$, as desired.
\end{proof}

\subsection{Invertibility of the transmutation operator}

Let $y\in W^{1,p}(-\ell,\ell)$. Suppose that $u\in W^{1,p}(-\ell,\ell)$ satisfies the equation $y=\mathbf{T}_au$. Integration by parts yields
\begin{align*}
	y(x) & = u(x)-K_a(x,t)u'(t)\big{|}_{-x}^{x}+\int_{-x}^{x}\partial_tK_a(x,t)u(t)dt \\
	&= u(x)-K_a(x,x)u(x)+K_a(x,-x)u(-x)+\int_{-x}^{x}\partial_t K_a(x,t)u(t)dt\\
	&= \frac{u(x)}{a(x)}+\int_{-x}^{x}\partial_tK_a(x,t)u(t)dt,
\end{align*}
where the last equality follows from the Goursat conditions \eqref{eq:Goursatcondition}. This leads to the Volterra integral equation
\[
a(x)y(x)= u(x)+\int_{-x}^{x}a(x)\partial_tK_a(x,t)u(t)dt
\]
where the kernel $a\partial_tK_a\in L^{\infty}(\Omega)$. Since $ay\in L^{\infty}(-\ell,\ell)$, this Volterra equation of the second kind admits a unique solution $u\in L^{\infty}(-\ell,\ell)$ given by
\begin{equation}\label{eq:invertibilityrelation}
	u(x)=a(x)y(x)+\int_{-x}^{x}L(x,t)[a(t)y(t)]dt,
\end{equation}
where $L\in L^{\infty}(R)$ (see, e.g., \cite[Ch. 10, Sub. 2.5]{kolmogorov} or \cite[Ch. IV]{vladimirov}). In particular, we obtain that $\mathbf{T}_au=0$ implies $u=0$, i.e., the operator $\mathbf{T}_a$ is injective.
We denote by $\mathbf{V}_a$ the operator on the right-hand side of \eqref{eq:invertibilityrelation}. It is the composition of a Volterra integral operator with kernel $L \in L^{\infty}(R)$ and the multiplication operator by $a$. Hence, $\mathbf{V}_a \in \mathcal{B}(L^p(-\ell,\ell))$. Since $\mathbf{T}_a[x^k] = \varphi_a^{(k)}(x)$ for all $k \in \mathbb{N}_0$, we obtain
\[
x^k=\mathbf{V}_a\varphi_a^{(k)}(x),\quad \forall k\in \mathbb{N}_0.
\]

Consequently, if $p(x) = \sum_{k=0}^{N} \alpha_k x^k$, then
\[
p(x) = \sum_{k=0}^{N} \alpha_k \mathbf{V}_a \varphi_a^{(k)}(x).
\]
Therefore, the operators $\mathbf{T}_a: \Pi[-\ell,\ell] \to \Pi_a[-\ell,\ell]$ and $\mathbf{V}_a: \Pi_a[-\ell,\ell] \to \Pi[-\ell,\ell]$ are linear isomorphisms, and they satisfy
\begin{equation}\label{eq:inverseoperator}
	\mathbf{T}_a \mathbf{V}_a = \mathbf{I}_{\Pi_a[-\ell,\ell]} \quad \text{and} \quad \mathbf{V}_a \mathbf{T}_a = \mathbf{I}_{\Pi[-\ell,\ell]}.
\end{equation}

\begin{lemma}
	There exists $C>0$ such that
	\begin{equation}\label{eq:invertibilityinequality}
		\|\mathbf{T}_ap\|_{W^{1,p}(-\ell,\ell)}\geq C\|p\|_{W^{1,p}(-\ell,\ell)}\qquad \forall p\in \Pi[-\ell,\ell].
	\end{equation}
\end{lemma}
\begin{proof}
	Let $p(x)=\sum_{k=0}^{N}\alpha_kx^k$. Since $\mathbf{V}_a\in \mathcal{B}(L^p(-\ell,\ell))$, we have
	\begin{equation*}
		\|p\|_{L^p(-\ell,\ell)}=\|\mathbf{V}_a\mathbf{T}_ap\|_{L^p(-\ell,\ell)}\leq \|\mathbf{V}_a\|_{\mathcal{B}(L^p(-\ell,\ell))}\|\mathbf{T}_ap\|_{L^p(-\ell,\ell)}.
	\end{equation*}
	To analyze the derivative, consider the operator $\mathbf{V}_{\frac{1}{a}}\in \mathcal{B}(L^p(-\ell,\ell))$. From equation \eqref{eq:derivativeformalpower}, for $k\geq 1$, we obtain
	\begin{align*}
		Dx^k=kx^{k-1}=k\mathbf{V}_{\frac{1}{a}}\varphi_{\frac{1}{a}}^{(k-1)}(x)=\mathbf{V}_{\frac{1}{a}}\mathbf{D}_a\varphi_a^{(k)}.
	\end{align*}
	Since $D1=0=\mathbf{V}_{\frac{1}{a}}\mathbf{D}_a\varphi_a^{(0)}$, it follows  by linearity that
	\[
	Dp=\mathbf{V}_{\frac{1}{a}}\mathbf{D}_a\mathbf{T}_ap.
	\]
	Consequently,
	\[
	\|Dp\|_{L^p(-\ell,\ell)}\leq \|\mathbf{V}_{\frac{1}{a}}\|_{\mathcal{B}(L^p(-\ell,\ell))}\|\mathbf{D}_a\mathbf{T}_ap\|_{L^p(-\ell,\ell)}\leq \|\mathbf{V}_{\frac{1}{a}}\|_{\mathcal{B}(L^p(-\ell,\ell))}\|a^2\|_{L^{\infty}(-\ell,\ell)}\|D\mathbf{T}_ap\|_{L^p(-\ell,\ell)}.
	\]
	Hence
	\[
	\|p\|_{W^{1,p}(-\ell,\ell)} \leq M \|\mathbf{T}_ap\|_{W^{1,p}(-\ell,\ell)},
	\]
	where $M=\max\{\|\mathbf{V}_a\|_{\mathcal{B}(L^p(-\ell,\ell))},\|\mathbf{V}_{\frac{1}{a}}\|_{\mathcal{B}(L^p(-\ell,\ell))}\|a^2\|_{L^{\infty}(-\ell,\ell)}\}$. Therefore, inequality  \eqref{eq:invertibilityinequality} holds with constant $C=\frac{1}{M}$.
\end{proof}

\begin{theorem}
	For $1\leq p<\infty$, $\mathbf{T}_a\in \mathcal{G}(W^{1,p}(-\ell,\ell))$.
\end{theorem}
\begin{proof}
	Let $\widehat{p}\in \Pi_a[-\ell,\ell]$. Setting $p=\mathbf{V}_a\widehat{p}\in \Pi[-\ell,\ell]$, inequality \eqref{eq:invertibilityinequality} implies
	\[
	\|\mathbf{V}_a\widehat{p}\|_{W^{1,p}(-\ell,\ell)}=\|p\|_{W^{1,p}(-\ell,\ell)}\leq \frac{1}{C}\|\mathbf{T}_ap\|_{W^{1,p}(-\ell,\ell)}=\frac{1}{C}\|\widehat{p}\|_{W^{1,p}(-\ell,\ell)}.
	\]
	Hence, the operator $\mathbf{V}$ is bounded with respect to the $W^{1,p}$-norm on the subspace $\Pi_a[-\ell,\ell]$. Since $\Pi_a[-\ell,\ell]$ is dense in $W^{1,p}(-\ell,\ell)$  by Theorem \ref{Th:Completenessformalpowerssobolev}, the operator $\mathbf{V}_a$ admits a unique bounded extension to the whole space $W^{1,p}(-\ell,\ell)$. Abusing notation, we denote this extension by $\mathbf{V}_a$.  By the relations in \eqref{eq:inverseoperator}, Corollary  \ref{coro:polynomialaprox}, and Theorem \ref{Th:Completenessformalpowerssobolev}, we conclude that $\mathbf{T}_a\in \mathcal{G}(W^{1,p}(-\ell,\ell)) $ with inverse $\mathbf{T}_a^{-1}=\mathbf{V}_a$.
\end{proof}

\subsection{Relations between $\mathbf{T}_a$ and the Darboux associated operator $\mathbf{T}_{\frac{1}{a}}$}
Now we establish a relation between the operator $\mathbf{T}_a$ of Eq. \eqref{eq:impedanceeq}
and the operator $\mathbf{T}_{\frac{1}{a}}$ associated with the corresponding Darboux equation.

\begin{theorem}\label{Th:RelationDarboux1}
	The following identity holds:
	\begin{equation}\label{eq:relationTaTdarboux}
		\mathbf{D}_a\mathbf{T}_au=\mathbf{T}_{\frac{1}{a}}Du\qquad \forall u\in W^{2,p}(-\ell,\ell).
	\end{equation}
\end{theorem}
\begin{proof}
	By Eq. \eqref{eq:derivativeformalpower} and the fact that $\mathbf{D}_a\mathbf{T}_a1=0=\mathbf{T}_{\frac{1}{a}}D1$, the identity \eqref{eq:relationTaTdarboux} holds for all $p\in \Pi[-\ell,\ell]$. Since both operators $\mathbf{D}_a\mathbf{T}_a,\mathbf{T}_{\frac{1}{a}}D: W^{1,p}(-\ell,\ell)\rightarrow L^p(-\ell,\ell)$ are bounded, and since $\Pi[-\ell,\ell]$ is dense in $W^{2,p}(-\ell,\ell)$, the result follows by continuity.
\end{proof}

\begin{remark}
	Note that $\mathbf{T}_{\frac{1}{a}}D\in \mathcal{B}(W^{2,p}(-\ell,\ell),W^{1,p}(-\ell,\ell))$. Consequently, relation \eqref{eq:relationTaTdarboux} implies that $\mathbf{D}_a\mathbf{T}_{\frac{1}{a}}\in \mathcal{B}(W^{2,p}(-\ell,\ell),W^{1,p}(-\ell,\ell))$.
\end{remark}

\begin{corollary}
	The transmutation operators $\mathbf{T}_a$ and $\mathbf{T}_{\frac{1}{a}}$ satisfy the relation
	\begin{equation}
		\mathbf{T}_au=\mathbf{J}_a\mathbf{T}_{\frac{1}{a}}Du+u(0)\qquad \forall u\in W^{2,p}(-\ell,\ell).
	\end{equation}
\end{corollary}
\begin{proof}
	The result follows from Eq. \eqref{eq:relationTaTdarboux} and the inversion formula for $\mathbf{D}_{\frac{1}{a}}$.
\end{proof}
\newline

From Theorem \eqref{Th:RelationDarboux1}, we obtain a second proof of the transmutation property \eqref{eq:transmutationproperty}.
\newline

\begin{proof}[2nd. Proof of Th. \ref{Th:transmutationop}]
	Let $u\in W^{3,p}(-\ell,\ell)$. Then both $\mathbf{D}_a\mathbf{T}_au$ and $\mathbf{T}_{\frac{1}{a}}Du$ belong to $W^{1,p}(-\ell,\ell)$, which allows the application of operator $\mathbf{D}_a$. Thus,
	\begin{align*}
		\mathbf{L}_a\mathbf{T}_a u &= \mathbf{D}_{\frac{1}{a}}\left(\mathbf{D}_a\mathbf{T}_au\right) =\mathbf{D}_{\frac{1}{a}}\left(\mathbf{T}_{\frac{1}{a}}Du\right)=\left(\mathbf{D}_{\frac{1}{a}}\mathbf{T}_{\frac{1}{a}}\right)Du=\mathbf{T}_aD^2u.
	\end{align*}
\end{proof}

\begin{proposition}
	For every $m\in \mathbb{N}$, the following relation holds:
	\begin{equation}\label{eq:transmutationgeneralized}
		\mathbf{D}_a^{(m)}\mathbf{T}_au=\mathbf{T}_{a^{(-1)^m}}D^mu,\quad \forall u \in W^{m+1,p}(-\ell,\ell).
	\end{equation}
\end{proposition}
\begin{proof}
	Theorems \ref{Th:RelationDarboux1} and \ref{Th:transmutationop} imply that \eqref{eq:transmutationgeneralized} holds for $m=1,2$. We now prove the general case by induction on $m$. Suppose that \eqref{eq:transmutationgeneralized} is valid for $m\geq 1$. Given $u\in W^{m+2,p}(-\ell,\ell)$, by the induction hypothesis, $\mathbf{T}_au$ belongs to the domain of $\mathbf{D}_a^{(m)}$ and $\mathbf{D}_a^{(m)}\mathbf{T}_au=\mathbf{T}_{a^{(-1)^m}}D^mu\in W^{1,p}(-\ell,\ell)$, so we can apply  $(a^2)^{{(-1)}^m}D$  to obtain
	\begin{align*}
		\mathbf{D}_a^{m+1}\mathbf{T}_au &= (a^2)^{{(-1)}^m}D\mathbf{D}_a^m\mathbf{T}_au =(a^2)^{{(-1)}^m}D\mathbf{T}_{a^{(-1)^m}}D^mu
	\end{align*}
	Note that equality \eqref{eq:relationTaTdarboux} can be rewritten as $(a^2)^{{(-1)}^m}D\mathbf{T}_{a^{(-1)^m}}=\mathbf{T}_{a^{(-1)^{m+1}}}D$ on $W^{2,p}(-\ell,\ell)$. Since $D^mu\in W^{2,p}(-\ell,\ell)$, it follows that
	\[
	\mathbf{D}_a^{m+1}\mathbf{T}_au=\mathbf{T}_{a^{(-1)^{m+1}}}D^{m+1}
	\]
\end{proof}

\subsection{Relations between the kernels}
Transmutation relation \eqref{eq:relationTaTdarboux} induces certain differential relations between the transmutation kernel $K_a$ and the kernel $K_{\frac{1}{a}}$ associated with the Darboux equation.

\begin{theorem}\label{Th:kernelsrelations}
	The transmutation kernels $K_a(x,t)$ and $K_{\frac{1}{a}}(x,t)$ satisfy the following relations in $\mathcal{R}=\{(x,t)\in \mathbb{R}^2\, |\, -\ell\leq x\leq \ell,\; |t|\leq |x|\}$:
	\begin{equation}\label{eq:kernelsrelation}
		\partial_xK_{\frac{1}{a}}(x,t)= -a^2(x)\partial_tK_a(x,t)\quad\text{and}\quad \partial_xK_a(x,t)=-\frac{1}{a^2(x)}\partial_tK_{\frac{1}{a}}(x,t).
	\end{equation}
	
	Moreover, if $a\in C^1[-\ell,\ell]$, relations \eqref{eq:kernelsrelation} hold in the whole rectangle $\Omega$.
\end{theorem}
\begin{proof}
	It is enough to prove the first equality. Let $v\in C^2[-\ell,\ell]$. Using relations \eqref{eq:derivativetransmutation} and \eqref{eq:relationTaTdarboux}, we have
	\begin{align*}
		a^2(x)\mathbf{T}_aDv(x)& = D\mathbf{T}_{\frac{1}{a}}v(x)= a(x)v'(x)-\int_{-x}^{x}\partial_xK_{\frac{1}{a}}(x,t)v'(t)dt.
	\end{align*}
	Integration by parts on $\mathbf{T}_aDv(x)$ gives
	\begin{align*}
		\mathbf{T}_aDv(x) &= v'(x)-\int_{-x}^{x}K_a(x,t)v''(t)dt\\
		&= v'(x)-K_a(x,t)v'(t)\big{|}_{-x}^{x}+\int_{-x}^{x}\partial_tK_a(x,t)v'(t)dt\\
		&= \frac{v'(x)}{a(x)}+\int_{-x}^{x}\partial_tK_a(x,t)v'(t)dt,
	\end{align*}
	where the last equality follows from the Goursat conditions \eqref{eq:Goursatcondition}. Comparing both expression for $D\mathbf{T}_{\frac{1}{a}}v(x)=a^2(x)\mathbf{T}_aDv(x)$, we conclude that
	\begin{equation*}
		\int_{-x}^{x}a^2(x)\partial_tK_a(x,t)v'(t)dt=	-\int_{-x}^{x}\partial_xK_{\frac{1}{a}}(x,t)v'(t)dt.
	\end{equation*}
	This relation holds for every $v\in C^2[-\ell,\ell]$. Fix $x\in [-\ell,\ell]$, and let $\phi\in C_0^{\infty}(-\ell,\ell)$ be an arbitrary test function with support contained in $[-|x|,|x|]$. Taking $\Phi\in C^{\infty}[-\ell,\ell]$ with $\Phi'=\phi$, we obtain
	\begin{equation*}
		\int_{-x}^{x}a^2(x)\partial_tK_a(x,t)\phi(t)dt=-\int_{-x}^{x}\partial_xK_{\frac{1}{a}}(x,t)\phi(t)dt.
	\end{equation*}
	Since $\phi$ was arbitrary, it follows that $\partial_xK_{\frac{1}{a}}(x,t)=-a^2(x)\partial_tK_a(x,t)$ for all $|t|\leq |x|$. Because $x\in [-\ell,\ell]$ was arbitrary, we conclude that \eqref{eq:kernelsrelation} holds, as desired.
	
	Now, assume that $a\in C^1[-\ell,\ell]$.	Let $v\in C^2[-\ell,\ell]$. According to \eqref{eq:relationTaTdarboux}, we have $\mathbf{D}_{\frac{1}{a}}\mathbf{T}_{\frac{1}{a}}v=\mathbf{T}_aDv$. Since both sides of the equality belong to $W^{1,p}(-\ell,\ell)$, we can apply the inverses of $\mathbf{T}_a$ and $\mathbf{T}_{\frac{1}{a}}$ to obtain
	\[
	\mathbf{T}_a^{-1}\mathbf{D}_{\frac{1}{a}}v(x)= D\mathbf{T}_{\frac{1}{a}}^{-1}v(x),\quad -\ell\leq x\leq \ell.
	\]
	By \cite[Th. 30]{mineimpedance1}, since $a\in C^1[-\ell,\ell]$, the inverse $\mathbf{T}_a^{-1}$ can be represented as
	\begin{equation}\label{eq:inverseasmooth}
		\mathbf{T}_av(x)=v(x)-\int_{-x}^{x}K_{\frac{1}{a}}(t,x)v'(t)dt.
	\end{equation}
	Applying Leibniz's rule and using the Goursat conditions \eqref{eq:Goursatcondition}, we get
	\begin{equation}\label{eq:auxiliar6}
		D\mathbf{T}_{\frac{1}{a}}^{-1}v(x)= \frac{v'(x)}{a(x)}-\int_{-x}^{x}\partial_x(K_a(t,x))v'(t)dt.
	\end{equation}
	On the other hand, integration by parts and using the Goursat conditions again, we obtain
	\begin{align*}
		\mathbf{T}_a^{-1}\mathbf{D}_{\frac{1}{a}}v(x) & = \frac{v'(x)}{a^2(x)}-\int_{-x}^{x}K_{\frac{1}{a}}(t,x)\frac{d}{dt}\left(\frac{v'(t)}{a^2(t)}\right)dt\\
		&= \frac{v'(x)}{a^2(x)}-K_{\frac{1}{a}}(t,x)\frac{v'(t)}{a^2(t)}\bigg{|}_{-x}^{x}+\int_{-x}^{x}\partial_t(K_{\frac{1}{a}}(t,x))\frac{v'(t)}{a^2(t)}dt\\
		&= \frac{v'(x)}{a^2(x)}-(1-a(x))\frac{v'(x)}{a^2(x)}+\int_{-x}^{x}\partial_t(K_{\frac{1}{a}}(t,x))\frac{v'(t)}{a^2(t)}dt\\
		&= \frac{v'(x)}{a(x)}+\int_{-x}^{x}\partial_t(K_{\frac{1}{a}}(t,x))\frac{v'(t)}{a^2(t)}dt.
	\end{align*}
	Comparing this expression with \eqref{eq:auxiliar6} yields
	\begin{equation*}
		\int_{-x}^{x}\frac{\partial_t(K_{\frac{1}{a}}(t,x))}{a^2(t)}v'(t)dt=-\int_{-x}^{x}\partial_x(K_a(t,x))v'(t)dt,
	\end{equation*}
	for all $-\ell\leq x\leq \ell$ and $v\in C^2[-\ell,\ell]$. Applying a procedure analogous to the first part of the proof, we obtain
	\begin{equation*}
		\frac{\partial_t(K_{\frac{1}{a}}(t,x))}{a^2(t)}=-\partial_x(K_a(t,x))\quad \text{for } \;\;\; |t|\leq |x|\leq \ell.
	\end{equation*}
	A straightforward computation shows that this equation is equivalent to
	\[
	\frac{\partial_xK_{\frac{1}{a}}(x,t)}{a^2(x)}=-\partial_tK_a(x,t)\quad \text{for}\;\;\; |x|\leq |t|\leq \ell.
	\]
	Therefore, by \eqref{eq:kernelsrelation}, this relation holds throughout the entire rectangle $\Omega$.
\end{proof}

System \eqref{eq:kernelsrelation} is also known as a generalized hyperbolic Cauchy-Riemann system, and its solutions are referred to as $a^2$-hyperbolic analytic \cite{pakhareva}. Actually, this system can be reformulated in terms of the so-called {\bf hyperbolic main Vekua equation} \cite{kravchenkorochontrembley}. From this, it is possible to obtain a representation of $K_{\frac{1}{a}}$ in terms of $K_a$.

\begin{proposition}
	If $a\in C^1[-\ell,\ell]$, then the following relation holds:
	\begin{equation}\label{eq:integralrelationdarbouxkernel}
		K_{\frac{1}{a}}(x,t)=-\int_0^xa^2(\xi)\partial_{\zeta} K_a(\xi,0)d\xi-\int_0^ta^2(x)\partial_{\xi}K_a(x,\zeta)d\zeta,\quad (x,t)\in \Omega.
	\end{equation}
\end{proposition}
\begin{proof}
	First, assume that $a\in C^2[-\ell,\ell]$. The system \eqref{eq:kernelsrelation} can be written as
	\[
	\nabla K_{\frac{1}{a}}(x,t)=(P(x,t),Q(x,t)) \text{ where  }\; P(x,t)=-a^2(x)\partial_tK_a(x,t), \; Q(x,t)=-a^2(x)\partial_xK_a(x,t).
	\]
	Since $a\in C^2[-\ell,\ell]$, according to \cite[Prop. 2]{mineimpedance1}, the kernel $K_a\in C^2(\overline{\Omega})$ and satisfies the equation $\partial_x(a^2(x)\partial_xK_a)=a^2(x)\partial_t^2K_a$ in $\Omega$. Hence
	\[
	\partial_xQ-\partial_tP= -\partial_x(a^2(x)\partial_xK_a)+a^2(x)\partial_t^2K_a=0.
	\]
	Since $\Omega$ is simply connected, the vector field $(P,Q)$ is conservative, and $K_{\frac{1}{a}}$ can be recovered as
	\[
	K_a(x,t)=\int_{\Gamma}\left\{-a^2(\xi)\partial_{\zeta}K_a(\xi,\zeta)d\xi-a^2(\xi)\partial_{\xi}K_a(\xi,\zeta)d\zeta \right\}+c
	\]
	where $c$ is a constant and $\Gamma$ is a rectifiable path joining  $(0,0)$ with $(x,t)$. Taking the polygonal path from $(0,0)$ to $(x,0)$ and then to $(x,t)$, and using that  $K_a(0,0)=0$, we obtain \eqref{eq:integralrelationdarbouxkernel}.
	
	For the case $a\in C^1[-\ell,\ell]$, consider the function $q=\frac{2a'}{a}\in C[-\ell,\ell]$. Take a sequence $\{q_n\}_{n=0}^{\infty}\in C^2[-\ell,\ell]$ such that $q_n\rightarrow q$ uniformly on $[-\ell,\ell]$. Define $a_n(x)=e^{-\frac{1}{2}\int_0^xq_n(s)ds}$. By \cite[Th. 6]{mineimpedance1}, the corresponding kernels $K_{a_n}$ converge uniformly on $\Omega$ together with their first derivatives to $K_a$.  Passing to the limit in the integral formula yields \eqref{eq:integralrelationdarbouxkernel}.
\end{proof}

\section{Transmutation operator for $a\in W^{1,2}$}

First, assume that $a\in W^{1,1}(0,\ell)$. By \eqref{eq:initialconditionssole}-\eqref{eq:conditionssine}, $e_a(\rho,x)=C_a(\rho,x)+i\rho S_a(\rho,x)$. According to \cite[Th. 22]{carroll} and \cite[Prop. 1]{mineimpedance1}, for fixed $x\in (0,\ell]$, the solutions $C_a(\rho,x)$ and $S_a(\rho,x)$ satisfy the following estimates for all $\rho\in \mathbb{C}$, :
\[
|C_a(\rho,x)-\cos(\rho x)|\leq \frac{c|\rho|xe^{|\operatorname{Im}\rho|x}}{1+|\rho|x}e^{Q_1(x)},\quad \left|S_a(\rho,x)-\frac{\sin(\rho x)}{\rho}\right|\leq \frac{cxe^{|\operatorname{Im}\rho|x}}{1+|\rho|x}e^{Q_1(x)}
\]
where $Q_1(x):=\int_0^x|q(t)|dt$ and $c$ is a constant independent of $x$ and $\rho$. From these estimates, we obtain 
\[
|e_a(\rho,x)-e^{i\rho x}|\leq \frac{2c|\rho|xe^{|\operatorname{Im}\rho|x}}{1+|\rho|x}e^{Q_1(x)} \quad \forall \rho \in \mathbb{C}.
\]
Consider the function $\varphi_x(\rho):=\frac{e^{i\rho x}-e_a(\rho,x)}{i\rho}$. Hence $\varphi_x\in \operatorname{Hol}(\mathbb{C})$ and is of exponential type $x$. Moreover, the following estimate holds:
\[
|\varphi_x(\rho)|\leq \frac{cxe^{|\operatorname{Im}\rho|x}}{1+|\rho|x}e^{Q_1(x)},
\]
from which we deduce that $\varphi_x|_{\mathbb{R}}\in L^2(\mathbb{R})$. By the Paley-Wiener theorem \cite[Th. 19.3]{rudin}, it follows that $\varphi_x(\rho)=\int_{-x}^xk_x(t)e^{i\rho t}dt$, for some $k_x\in L^2(-x,x)$. Defining $K_a(x,t):=k_x(t)$, we obtain the representation \eqref{eq:integralsole}. The kernel $K_a(x,t)$ is defined on the triangle $\mathcal{T}:=\{(x,t)\in \mathbb{R}^2\, |\, |t|\leq x\leq \ell\}$. Furthermore, by Plancherel's theorem
\begin{align*}
	\int_{-x}^{x}|K_a(x,t)|^2dt & = \frac{1}{2\pi}\int_{-\infty}^{\infty}|\varphi_x(\rho)|^2d\rho \\
	& \leq \frac{c^2x^2e^{2\|q\|_{L^1(0,\ell)}}}{2\pi}\int_{-\infty}^{\infty}\frac{d\rho}{(1+|\rho|x)^2}\\
	&= \frac{c^2xe^{2\|q\|_{L^1(0,\ell)}}}{\pi}\int_0^{\infty}\frac{dt}{(1+t)^2}\\
	& \leq \frac{c^2\ell e^{2\|q\|_{L^1(0,\ell)}}I}{\pi},
\end{align*}
where $I=\int_0^{\infty}\frac{dt}{(1+t)^2}$. Hence
\begin{align*}
	\iint_{\mathcal{T}}|K_a(x,t)|^2dtdx = \int_0^{\ell}\int_{-x}^{x}|K_a(x,t)|^2dtdx\leq \frac{c^2\ell^2e^{2\|q\|_{L^1(0,\ell)}}I}{\pi}<\infty.
\end{align*}
In summary:
\begin{proposition}
	If $a\in W^{1,1}(0,\ell)$, the solution $e_a(\rho,x)$ admits the representation \eqref{eq:integralsole} for all $x\in (0,\ell]$, with the integral kernel $K_a\in L^2(\mathcal{T})$. 
\end{proposition}
In particular, the result is valid for $a\in W^{1,2}(0,\ell)$. In this case, we consider the operator $\mathbf{T}_a$ given by \eqref{eq:transmutationoperator}, acting from $W^{1,2}(-\ell,\ell)$ into $L^2(0,\ell)$. Since $K_a\in L^2(\mathcal{T})$, it follows that $\mathbf{T}_a\in \mathcal{B}(W^{1,2}(-\ell,\ell),L^2(0,\ell))$. Furthermore, due to the uniform convergence of the Taylor series of $e^{i\rho x}$, repeating the proof of Proposition \ref{Prop:mappingproperty} shows that the mapping property \eqref{eq:mappingproperty} continues to hold.

\begin{theorem}
	Suppose that $a\in W^{1,2}(0,\ell)$. Then the following statements hold:
	\begin{itemize}
		\item[(i)] Relation \eqref{eq:relationTaTdarboux} is valid for $u\in W^{2,2}(-\ell,\ell)$.
		\item[(ii)] The transmutation property \eqref{eq:transmutationproperty} is valid for $u\in W^{3,2}(-\ell,\ell)$. 
	\end{itemize}
\end{theorem}
\begin{proof}
	It suffices to prove (i) (the property (ii) follows from (i) y the same arguments as in the second proof of Theorem \ref{Th:transmutationop}). From the relations \eqref{eq:mappingproperty} and  \eqref{eq:derivativeformalpower}, together with the linearity of $\mathbf{T}_a$, we have $\mathbf{D}_a\mathbf{T}_ap=\mathbf{T}_{\frac{1}{a}}p'$ for $p\in \Pi[-\ell,\ell]$, and since $\mathbf{T}_ap(0)=p(0)$, it follows that
	\[
	\mathbf{T}_ap=\mathbf{J}_a\mathbf{T}_{\frac{1}{a}}p'+p(0)\qquad \forall p\in \Pi[-\ell,\ell].
	\]
	Let $u\in W^{2,2}(-\ell,\ell)$. By Corollary \ref{coro:polynomialaprox}, there exists a sequence $\{p_n\}_{n=0}^{\infty}\subset \Pi[-\ell,\ell]$ such that $p_n^{(j)}\rightarrow u^{(j)}$ uniformly on $[-\ell,\ell]$ for $j=0,1$, and $p_n''\rightarrow u''$ in $L^2(-\ell,\ell)$. Since $p_n\rightarrow u$ in $W^{1,2}(-\ell,\ell)$, by continuity of $\mathbf{T}_a$ it follows that  $\mathbf{T}_ap_n\rightarrow \mathbf{T}_au$ in $L^2(0,\ell)$. On the other hand, $p_n'\rightarrow u'$ in $W^{1,2}(-\ell,\ell)$, so $\mathbf{J}_a \mathbf{T}_{\frac{1}{a}}p_n'\rightarrow \mathbf{J}_a\mathbf{T}_{\frac{1}{a}}u$, and since $p_n(0)\rightarrow u(0)$, we conclude that $\mathbf{T}_au=\mathbf{J}_a\mathbf{T}_{\frac{1}{a}}u'+u(0)$, from where we obtain \eqref{eq:relationTaTdarboux}.
\end{proof}

\section{Conclusions}

The construction of a complete system of functions, called formal powers, associated with the Sturm-Liouville equation with an integrable impedance, is presented. We establish the relationship between the formal powers corresponding to the original Sturm-Liouville equation and those of its Darboux-transformed counterpart. Moreover, approximation results via generalized Taylor polynomials and completeness in $L^p$ and Sobolev spaces are proved.

The completeness of the formal powers enables us to derive key properties for a pair of transmutation operators connecting the Sturm-Liouville equation and its Darboux transform, including the existence of continuous inverses. These results provide powerful tools for the analytical representation of solutions to the Sturm-Liouville equation, as well as for addressing direct and inverse spectral boundary value problems. This framework is exemplified in the case of equations such as the Schrödinger equation in \cite{kravchenkolibro,NSBF1, kravchenkotorbainverse}  
\newline

{\bf Acknowlegdments:}
The author thanks to Instituto de Matemáticas de la U.N.A.M. Unidad Querétaro (México), where this
work was developed, and the SECIHTI for their support through the program {\it Estancias Posdoctorales por
	México Convocatoria 2023 (I)}.
\newline

{\bf Conflict of interest:} This work does not have any conflict of interest.


\end{document}